\documentclass[preprint,12pt]{amsart}

\usepackage{amssymb}
\usepackage{amsmath}
 \usepackage{amsthm}
\usepackage[T1]{fontenc}
\usepackage{xcolor}
\usepackage{enumerate}
\usepackage{url}
\usepackage[colorlinks, linkcolor=red, anchorcolor=blue, citecolor=green]{hyperref} 

\usepackage{geometry}
\allowdisplaybreaks[1]

\usepackage{cleveref}
\newtheorem{theorem}{Theorem}[section]

\newtheorem{lemma}[theorem]{Lemma}
\newtheorem{proposition}[theorem]{Proposition}
\newtheorem{remark}[theorem]{Remark}

\newtheorem{corollary}[theorem]{Corollary}
\newtheorem*{theorem*}{Theorem} 
\newtheorem*{corollary*}{Corollary} 

\newtheorem{assumptions}[theorem]{Assumptions}
\newtheorem{Notations}[theorem]{Notations}

\numberwithin{equation}{section}

\newcommand{\supp}{\operatorname{supp}} 
\newcommand{\e}{{\rm e}} 

\newcommand{\mfa}{\operatorname{\mathfrak{a}}}

\newcommand{\mfg}{\operatorname{\mathfrak{g}}}

\newcommand{\mfk}{\operatorname{\mathfrak{k}}}

\newcommand{\mfm}{\operatorname{\mathfrak{m}}}
\newcommand{\mfn}{\operatorname{\mathfrak{n}}}

\newcommand{\mfp}{\operatorname{\mathfrak{p}}}

\newcommand{\mfw}{\operatorname{\mathfrak{w}}}

\newcommand{\R}{\operatorname{\mathbb{R}}}
\newcommand{\C}{\operatorname{\mathbb{C}}}
\newcommand{\Z}{\operatorname{\mathbb{Z}}}
\newcommand{\N}{\operatorname{\mathbb{N}}}

\let\eps\varepsilon
\def\D{\mathbb{D}}

\def\abso #1{ \left| #1 \right| } 
\def\set #1{ \left\{ #1 \right\} } 
\def\norm #1{ \left\| #1 \right\| } 
\def\bracket #1{ \left( #1 \right) } 
\def\Bracket #1{ \left[ #1 \right] } 
\def\ip #1{ \left\langle #1 \right\rangle } 

\newcommand{\HCc}{\operatorname{\mathbf{c}}} 
\newcommand{\LB}{{F(\Delta_\rho)}}
\newcommand{\La}{{F(\mathcal{L})}} 
\newcommand{\CA}{{\text{Cl($A^+$)}}} 
\newcommand{\Ca}{{\text{Cl($\mfa^+$)}}} 
\let\a\alpha
\let\l\lambda
\let\phi\varphi
\def\cal#1{{\mathcal#1}}

\begin{document}


\title{Shifted wave equation on noncompact symmetric spaces} 

\author[Y. Kuznetsova]{Yulia Kuznetsova}
\address{Yulia Kuznetsova
\endgraf
Universit\'e Marie et Louis Pasteur, CNRS, LmB (UMR 6623), F-25000 Besançon, France}
\email{yulia.kuznetsova@univ-fcomte.fr}

\author[Z. Song]{Zhipeng Song}
\address{Zhipeng Song
\endgraf
Ghent University, Department of Mathematics: Analysis, Logic and Discrete Mathematics, 9000 Ghent, Belgium
\endgraf
Universit\'e Marie et Louis Pasteur, CNRS, LmB (UMR 6623), F-25000 Besançon, France
}
\email{zhipeng.song@univ-fcomte.fr}


\begin{abstract}
Let $G$ be a semisimple, connected, and noncompact Lie group with a finite center.  We carry out a detailed analysis of oscillating integrals involving the Harish-Chandra $c$-function, in the case of real rank $l\ge 2$. This allows to obtain two main applications. Consider the Laplace-Beltrami operator $\Delta$ on the homogeneous space $G/K=S$ by a maximal compact subgroup $K$. We obtain pointwise estimates for the kernel of an oscillating function $\exp( it\sqrt{|x|}) \psi(\sqrt{|x|}) $ applied to the shifted Laplacian $\Delta+|\rho|^2$. We obtain a polynomial decay in time of the kernel, and of the $L^p-L^q$ norms of the operator, for $1\le p<2<q\le \infty$. For the related distinguished Laplacian, we obtain bounds for the $L^p-L^p$ norms, $1\le p\le\infty$, with a slower growth in time than predicted by earlier results.
\end{abstract}




\maketitle

\section{Introduction}
Let $G$ be a semisimple, connected, and noncompact Lie group with a finite center. Its homogeneous space $G/K=S$ by a maximal compact subgroup $K$ is a Riemannian manifold, with its Laplace--Beltrami operator $\Delta$.

It is a self-adjoint operator on $L^2(S)$, $|\Delta|=-\Delta$ is positive, 
and by the spectral theorem, every bounded Borel function $F$ on $\R$ defines a bounded operator on $L^2(S)$:
\[
F(|\Delta|)=\int_0^\infty F(\xi) dE(\xi),
\] 
where $E$ is the spectral measure of $\Delta$. Moreover, being left-invariant, $F(|\Delta|)$ acts as a convolution on the right with a kernel $k_{F(|\Delta|)}$, a priori defined in the distributional sense. 

Much attention has been devoted to the following question: 
under what conditions 
$F(|\Delta|)$ is also bounded from $L^p(S)$ to $L^q(S)$?
%
The answer is well known for the heat semigroup 
$\exp(-t|\Delta|)$ (see the book 
\cite{var-book}, the survey \cite{
anker-survey}, as well as recent results \cite{
anker-heat-solutions}); 
the Poisson semigroup $\exp( -t \sqrt{|\Delta|} )$ \cite{anker1992sharp, AnkerJi, CGM1, CGM2, CGM3}, 
the resolvents $(z-\Delta)^{-s}$ \cite{anker1992sharp,AnkerJi, CGM1}.

For oscillating functions of the type
\begin{equation}\label{Psi-intro}
\Psi(x) = e^{it\sqrt{|x|}} \psi(\sqrt{|x|}),
\end{equation}
giving solutions of the wave equation, numerous estimates are known on symmetric spaces of rank one and on Damek-Ricci spaces (also of rank one but possibly non-symmetric) \cite{tataru,anker-hyperbolic,anker-damek}, see also references in \cite{anker-damek}. But in higher rank, where less explicit formulas are available, the results are few \cite{hassani,AnkerZhang,CGM3,CGMWave}. Important cases to mention are $L^p-L^q$ bounds for the Poisson kernel, that is $\psi(r)=e^{-u r}$, especially in small time~$t$ \cite{CGM3}, and for $\psi(r) = (\beta^2+ r^2)^{ -\alpha/2}$ at fixed time $t=1$ \cite{CGMWave}. We will return to them in the sequel.

Before going into further detail, let us mention that the situation is similar for the Schr\"odinger equation (see \cite{AnkerVittoriaVallarino} in rank one and \cite{anker-schr} in higher rank).

Our main interest in this paper is in functions $\Psi$ of the type \eqref{Psi-intro}.
By $K$-biinvariance, it is sufficient to estimate the kernel on the positive part of $G$ in the Cartan decomposition $G=K\CA K$. For $x\in G$, let $x^+$ denote the vector in $\Ca$ corresponding to this decomposition; moreover, let $n$ denote the dimension on $S$, $l$ its rank, $d$ the number of reduced roots and $\nu=2d+l$ the pseudo-dimension (see more on notations in Section \ref{sec-symmetric}). The statement involves also the functional $\rho$ equal to the half-sum of positive roots.
Hassani \cite{hassani}, and then Anker and Zhang \cite{AnkerZhang} with stronger estimates obtain the following theorem:
\begin{theorem*}[\cite{AnkerZhang}] Set $\psi(\sqrt x) = |x|^{-\a/2}$ with $\sigma\in \C$ and $\Re \,\a \le(n+1)/2$; 
then the kernel $k_t$ of $\Psi(\Delta)$ satisfies
\begin{equation}\label{kt-AZ}
|k_t(x)| \lesssim \begin{cases}
    |t|^{(1-n)/2}  \, (1+|x^+|)^s \,e^{ -\rho(x^+)}, & 0<|t|< 1, \ s=\frac{\max(n,\nu)}2-\frac l2;\\ 
    |t|^{ -\nu/2} \, (1+|x^+|)^{2d+l/2} \,e^{ -\rho(x^+)}, & |t|\ge1. 
    \end{cases}
\end{equation}
\end{theorem*}
The central part of the proof relies on the stationary phase analysis of the oscillating $l$-dimensional integral \eqref{integral formula of kernel} giving the value of the kernel. The main application is to the following dispersive estimates, and to Strichartz estimates which we do not cite here:
\begin{corollary*}[\cite{AnkerZhang}]
	 Suppose that $n\ge3$ and set $\a_0 = \max( \frac12 - \frac1p, \frac12 - \frac1{\tilde p} )$. Then for $2<p,\tilde p<\infty$ and $\a \ge (n+1)\a _0$
	\begin{equation}\label{kt-AZ dispersive}
	\| |\Delta| ^{-\a/2} \exp( it\sqrt{|\Delta|} ) \|_{L^{\tilde p'}(S) \to L^p(S) }\lesssim \begin{cases}
	|t|^{ -(n-1) \a_0} , & 0<|t|< 1, \\ 
	 |t|^{ -\nu/2} , & |t|\ge1. 
 	\end{cases}
	\end{equation}
\end{corollary*}

Methods of Anker and Zhang also apply to translated operators $-\Delta_q = \Delta + q^2$ if $0<q<|\rho|$. 
But one cannot follow the same path if $q=|\rho|$. We should recall the spectrum of $\Delta$ is $(-\infty,-|\rho|^2]$, and its translations with $q$ as above still have a spectral gap, while $-\Delta_\rho = \Delta+|\rho|^2$ has not. As for the kernels, the difference is that the phase in the oscillating integral has a family of non-isolated critical points, and standard stationary phase analysis cannot be held.

We therefore use a different strategy to deal with this boundary case $\Delta_\rho$. We obtain estimates for $\Psi(\Delta_\rho)$ with a wide range of functions $\psi$, decreasing quickly enough. Our main result is the following theorem, stated below in a shortened form; the full one is \Cref{eit-psi-LB-no-gap}.

\begin{theorem*}[\Cref{eit-psi-LB-no-gap}]
Let $\psi \in C^{(p)}([0,+\infty))$ be polynomially decreasing with its $p$ derivatives and has zero of order $b\le p-2$ at point $0$.
Then:
\begin{enumerate}
\item If $p=\nu$, then for every $x\in G$ and $t\ne0$,
$$
|k_t(x)| 
\lesssim |t|^{-\nu} ( 1+ |x^+| ) ^{ \nu+d} e^{-\rho( x^+) }.
$$
\item If $p=\lfloor l/2 \rfloor$, then for $t$ large enough and $|x^+| \le 3|t|$,
$$
|k_t(x)| \lesssim\, |t|^{(1-l)/2} (\log |t|)^d e^{-\rho(x^+)}.
$$
\item If $p= 2d + \lfloor l/2\rfloor$, then for $|t|$ large enough and $|x^+| > 3|t|$
\begin{align*}
| k_t(x) | &\lesssim |x^+|^{-d-l} 
e^{-\rho(x^+)}.
\end{align*}
\item There is a constant $\delta$ depending on the group only such that for $|t|$ large enough and $x$ with $|x^+| >\log |t|$ and $||t|-|x^+| \,| > \delta \log |t|$
\begin{align*}
| k_t(x) | \lesssim \big( |x^+|^{-d-b-l} &+ |x^+|^{(1-l)/2} ||x^+| -|t||^{-d-b-(l+1)/2} \big)
	\\&\times (\log |t|)^{2(2d+b+l)+1} e^{-\rho(x^+)}.
\end{align*}
\end{enumerate}
\end{theorem*}
The first estimate applies also to small $t$, though our main interest is in large time $t$.

Statement 1 of this theorem is analogous to \eqref{kt-AZ} above. Note the difference in the exponent between the two cases, $\nu$ against $\nu/2$. We conjecture that both are optimal, for small $x^+$, and reflect different behaviour of the two operators.

This statement alone is sufficient to prove dispersive estimates similar to the Corollary above: in Theorem \ref{dispersive-LB-no-gap}, we prove that
    \[
    \norm{\Psi(\Delta_\rho)}_{L^{p'}(S)\to L^{p}(S)} \lesssim |t|^{-\nu}
    \]
for $t\ne0$ and any $2< p< \infty$, with $p'$ being the conjugate exponent of $p$.

Estimates 2-4 in the theorem require much more work and are in fact motivated by the study of another related operator $\cal L$, usually called the distinguished laplacian.

Recall that the group $S$ is non-unimodular. Conjugating $\Delta_\rho$ by the modular function (see details in Section \ref{sec-dist-lapl}), we obtain the distinguished laplacian $\cal L$, and for any function $F$ their kernels are tightly related: $k_\LB=\delta^{1/2} k_\La$. The behaviour of the two operators is however quite different, and notably, $\La$ is bounded on $L^p(S)$, $1< p<\infty$, and is of weak type $(1,1)$, under mild conditions on~$F$ (existence and decrease of a finite number of derivatives, see \cite{Hebisch,cowling1994spectral,sikora2002spectral}), while for $F(\Delta)$ being bounded on $L^p(S)$, $p\ne2$ requires $F$ to be holomorphic in a stripe around the real axis \cite{clerc1974lp}.

The mentioned number of derivatives for the case of $\cal L$ was first estimated as $s_1 = \frac12 + \max( \frac{\nu}2, \frac{n}2) + \varepsilon$ \cite{Hebisch,cowling1994spectral}, to be later improved by Sikora to $s = \max( \frac{\nu}2, \frac{n}2) + \varepsilon$ \cite{sikora2002spectral}.
For oscillating functions as in \eqref{Psi-intro}, the above results imply that the weak $(1,1)$ norm of $\Psi(\cal L)$
 is bounded by $C t^{s}$ with
$s = \max( \frac{\nu}2, \frac{n}2) + \varepsilon$. This gives also a bound of the norm of $\|\Psi(\cal L)\|_{L^p(S) \to L^p(S)}$ for every $1< p< \infty$, by interpolation.

In his unpublished thesis \cite{gadzinski}, Gadzinki obtained a better estimate 
\[
\|k_{\Psi(\cal L)}\|_1 \lesssim |t|^{\frac{\nu-1}2}
\]
 for the special function $\Psi(x) = \cos( t x) \exp(-x^2)$, suggesting that oscillating functions behave better than general ones.
We confirm this conjecture in a strong way. Our estimates 2-4 imply the following theorem:
\begin{theorem*}
Suppose that $\psi$ satisfies the assumptions of Theorem \ref{eit-psi-LB-no-gap}.
Then for $|t|$ large enough,
$$
\| k_{\Psi(\cal L)}\|_1 \lesssim |t|^{(\nu-1)/2} (\log |t|)^{3d+3+(3l+1)/2}.
$$
\end{theorem*}

The power $\frac {\nu-1}2$, valid for oscillating functions (with enough regularity), is thus $\frac12$ smaller than the power $s$ above which has been shown by Sikora to be optimal for general functions. We conjecture that this new power is optimal in the oscillating case.



Let us outline the strategy of the proof. As mentioned, $k_t(x)$ is given by an integral \eqref{integral formula of kernel} over $\l\in \mfa$,
including the radial function $\Psi(|\l|^2)$ and a non-radial part which does not depend on $\Psi$. This allows to represent it, in polar coordinates, as
$$
k_t(x) = \int_0^\infty \Psi(r^2) \xi(r,x) dr,
$$
with a density $\xi$ involving elementary spherical functions $\varphi_\lambda$ and the $\HCc$-function. Known  approximations of
spherical functions allow us to reduce $\xi$ to an oscillating integral \eqref{xi(r,h)} with $r|x^+|$ as the parameter. This can be treated by the stationary phase method, a task occupying Section 3. Normally, this method is applied with a large parameter, but using exact estimates of remainders, we can treat any values of $r|x^+|$. 



In Section 3, the first result, \Cref{main theorem}, gives the main term and remainders in this analysis. The rest of the section is devoted then to thorough estimates of the terms and their derivatives, with a dependence on $r$.


In Section 4, we present the main result \Cref{eit-psi-LB-no-gap} concerning the kernel estimate. To prove it, we apply the machinery developed in 
Section 3.

Using the Kunze-Stein phenomenon, dispersive and $L^p-L^q$ estimates of the operator $\Psi(\Delta_\rho)$ can be directly deduced from the kernel estimate (Theorems \ref{dispersive-LB-no-gap} and \ref{Lp-Lq-LB}).

Section \ref{sec-poisson} is devoted to the Poisson semigroup $\cal P_\tau = \exp( -\tau \sqrt{|\Delta_\rho|})$, with complex $\tau$ such that $\Re\tau\ge0$. Pointwise and $L^p-L^q$ bounds were known before \cite{AnkerJi, CGM3}, but it was clear that they were close to optimal for $\tau$ real, or $|\tau|$ small. We obtain bounds for large $\tau$ with a polynomial decrease in $\Im\tau$, in Theorem \ref{Lp-Lq-Poisson}.

For the study of $\exp( -\tau \sqrt{\cal L})$, the bounds of Theorem \ref{eit-psi-LB-no-gap} are not sufficient. In Lemma \ref{lemma-poisson-ptwise}, we obtain more precise pointwise estimates at infinity.
In Section \ref{sec-L1}, we show how they imply similar estimates for general functions, which finally allows to prove that the kernel $k_{\Psi(\cal L)}$ is integrable and to estimate its norm, in Theorem~\ref{L-l1-norm}, and $L^p-L^p$ norms of $\Psi(\cal L)$, $1\le p\le\infty$, in Theorem~\ref{L-Lp-Lq-norms}.

\paragraph{\bf{Acknowledgment}}
This work is supported by the EIPHI Graduate School (contract ANR-17-EURE-0002). The second author is partially supported by the Methusalem programme of the Ghent University Special Research Fund (BOF), grant number 01M01021.

We are much obliged to Professor Waldemar Hebisch for enlightening discussions and for sharing the text of the thesis of Pzemyslav Gadzinski \cite{gadzinski}. The second author would like to thank his co-advisor, Professor Michael Ruzhansky, for his help and support during the course of this research.

\section{Notations and Preliminaries}
Throughout this paper, we shall use the convention in which $C$ (respectively $C_l$, $C_\psi$ etc.) always denotes a positive finite constant (respectively a positive constant depending on the parameter $l$, on the function $\psi$ etc.), which may vary from sentence to sentence. We denote by $m$ multi-indices in $\Z_+^l$ and, for convenience, use the same notation $m$ for the norm $|m|=m_1+\dots+m_l$ of $m$. We also use $\partial^m f(\lambda)$ to denote $\frac{\partial^m }{\partial \l_1^{m_1}\dots \partial \l_l^{m_l}}f(\l)$ shortly.
\subsection{Symmetric spaces}\label{sec-symmetric}
Let $G$ be a semisimple, connected Lie group with finite center and $\mfg$ the corresponding Lie algebra of $G$. 
Let $K$ be a maximal compact subgroup of $G$ and $\mfk$ the Lie algebra of $K$. We have the Cartan decomposition $\mfg = \mfk\oplus \mfp$ and the corresponding Cartan involution $\theta$ acting as 1 on $\mfk$ and -1 on $ \mfp$. 
Let $\mfa$ be a maximal abelian subspace of $\mfp$, this is, $\mfa$ is a Lie subalgebra of $\mfg$ such that $\mfa\subset \mfp$. 
We denote by $\mfa^*$ the real dual space of $\mfa$ and $\mfa_{\C}^*$ the dual of its complexification $\mfa_{\C}$. 

Let $\Sigma\subseteq \mfa^* $ be the set of restricted roots of the pair $(\mfg,\mfa)$ and $\mfw$ the associated Weyl group. By choosing a lexicographic ordering of the roots we can define the set of positive roots $\Sigma^+$. Then $\Sigma_r^+$(resp. $\Sigma_s^+$) be the set of reduced (resp. simple) roots. Note that $\Sigma_s^+ \subseteq \Sigma_r^+ \subseteq \Sigma^+ \subseteq \Sigma$ and $\Sigma_s^+$ is a basis of $\mfa^*$.

Setting $\mfn=\bigoplus_{\alpha\in \Sigma^+}\mfg_{\alpha}$ where $\mfg_\alpha$ is the root subspaces associated with $\alpha$, we obtain the Iwasawa decomposition $\mfg= \mfk\oplus \mfa\oplus \mfn$.
Let $A$ (resp. $N$) be the analytic subgroup of $G$ with the Lie algebra $ \mfa$ (resp. $ \mfn$). 
On the Lie group level, we get the Iwasawa decomposition $G=KAN$. 
Similarly, we also have the decomposition $G=NAK$. 
Let $H(g)$ (resp. $A(g)$) denote the unique $\mfa$-component of $g\in G$ in the decomposition $g=k\exp(H(g))n$ (resp. $g=n'\exp(A(g))k'$) with $k,k'\in K, n,n'\in N$.

The Killing form $B$ of $\mfg$ is positive definite on $\mfa$ and defines the inner product on $\mfa$ and $\mfa^*$ as well. We define the positive Weyl chamber of $\mfa$ associated with $\Sigma^+$ as 
\[
\mfa^+=\set{H\in \mfa \vert \alpha(H)> 0\ \text{for all} \ \alpha\in \Sigma^+}.
\]
Denote by $\Ca=\set{H\in \mfa: \alpha(H)\ge 0 \ \text{for all} \ \alpha\in \Sigma^+}$ the closure of $\mfa^+$ in $\mfa$. The polar decomposition of $G$ is given by $G=K\CA K$ where $\CA=\exp(\Ca)$, and we denote $x^+$ the $\Ca$-component of $x\in G$ in this decomposition.

By the Weyl vector $\rho$, we mean the half sum of positive roots $\alpha$ counted with their multiplicities $m_\alpha=\dim\mfg_\alpha$, i.e.,
$\rho(H)=\frac{1}{2}\sum_{\alpha\in \Sigma^+}m_\alpha \alpha(H),H\in \mfa$.
Now, for any $f\in C_c(G)$, we recall the integral formula for $K\CA K$ decomposition \cite[page 382]{helgason1962differential}:
\begin{align}\label{integral_formula_KAK}
    \int_G f(g)d g=C_G \int_K\int_{ \mfa^+}\int_K f(k_1\exp(H)k_2)d k_1D(H)d Hd k_2.
\end{align}
with
\begin{equation}\label{est-D}
D(H)=\prod_{\alpha\in \Sigma^+}\sinh^{m_\alpha} \alpha(H)\le C\left( \frac{|H|}{1+|H|}\right)^{n-l}e^{2\rho(H)}
 \le C e^{2\rho(H)},
\end{equation}
$H\in \overline{ \mfa^+}$.

The homogeneous space $G/K$ can be identified with the solvable Lie group $S=AN$ with the Lie algebra $\mfa\oplus \mfn$. 
The group $G$ is unimodular, while $S$ is not (unless $G=A$) and has the modular function given by 
\[
\delta(an)=\e^{-2\rho(\log a)}, \ \ an\in S.
\]
In addition to using $n$ to denote the dimension of $S$, we also introduce the pseudo-dimension of $S$, defined as $ \nu := 2d + l $, where $l$ is the dimension of $A$ and $d$ is the cardinality of the set $ \Sigma_r^+ $.
\subsection{Spherical Fourier analysis}
Harmonic analysis on symmetric spaces is built upon spherical functions.
It is important to note that our main references, Helgason \cite{helgason1962differential,helgason2000groups} and Gangolli and Varadarajan \cite{gangolli1988harmonic}, have different conventions in their notations. We chose to adopt Helgason's notations in our work; translation from \cite{gangolli1988harmonic} to \cite{helgason2000groups} is done by $\lambda \mapsto i\lambda$.

For each $\lambda\in \mfa_{\C}^*$, let $\varphi_\lambda(x)$ be the elementary spherical function given by 
\[
\varphi_\lambda(x)= \int_K e^{(i\lambda-\rho)(H(xk))} dk, \quad \forall x\in G.
\]
Using the identity $-H(g)=A(g^{-1})$ together with \cite[p.419, (7)]{helgason2000groups}, we obtain
\[
\varphi_\lambda(x)= \int_K e^{(-i\lambda+\rho)(A(kx^{-1}))} dk = \int_K e^{(i\lambda+\rho)(A(kx))} dk.
\]
These are exactly the joint eigenfunctions of all invariant differential operators on $G/K$.
We have the following functional equation: $\varphi_\lambda= \varphi_{s\lambda}$ for all $\lambda\in \mfa_{\C}^*, s\in \mfw$.
A basic estimate for $\varphi_\lambda$ is given as following:
\begin{proposition}\label{estimate of phi_lambda}
For all $\l\in \mfa^*, H\in \Ca$
\begin{equation}\label{est-phi_0}
|\phi_\l(e^H)| \le \phi_0(e^H) \le \e^{-\rho(H)}(1+\|H\|)^d.
\end{equation}
Moreover, for any multi-index $m\in \Z_+^l$ of norm $|m|=m_1+\dots+m_l$ there is a constant $C_m>0$ such that
    \[
    \abso{\frac{\partial^m}{\partial \l^m} \varphi_\l(e^H) }
    \le C_m (1+\|H\|)^{|m|} \, \phi_0(e^H).
    \]
\end{proposition}
\begin{proof}
The first statement is \cite[(4.6.3) and (4.6.9)]{gangolli1988harmonic} and the second one follows from \cite[(4.6.5) and Lemma 4.6.7]{gangolli1988harmonic}. 
\end{proof}

More information on spherical functions is given by the Harish-Chandra $\HCc$-function. An explicit expression for it as a meromorphic function on $\mfa_{\C}^*$ is given by the Gindikin-Karpelevich formula: For each $\lambda\in \mfa_{\C}^*$, 
\begin{equation}\label{gindikin-k}
\HCc(\lambda) = c_0 \prod_{\alpha\in \Sigma_r^+}
\frac{ 2^{-i\ip{ \lambda, \alpha_0}} \Gamma\big(i\ip{ \lambda, \alpha_0}\big)}{
\Gamma\big(\frac{1}{2} (\frac{1}{2}m_\alpha +1+ i\ip{ \lambda, \alpha_0})\big)
\Gamma\big(\frac{1}{2}(\frac{1}{2}m_\alpha +m_{2\alpha}+ i\ip{ \lambda, \alpha_0})\big)},
\end{equation}
where $\alpha_0=\alpha/\ip{ \alpha, \alpha }$, $\Gamma$ is the classical $\Gamma$-function, and $c_0$ is a constant given by $\HCc(-i\rho)=1$. 

From the formula above, one can deduce for $\lambda \in \mfa^*$: 
\begin{align}\label{estiofHCc}
    |\HCc(\lambda)|^{-2}
    \le C |\lambda|^{\nu-l} (1+|\lambda|)^{n-\nu}
    \le C 
    \begin{cases}
    |\lambda|^{\nu-l} & \text{if}\ |\lambda|\le 1,\\
    |\lambda|^{n-l} & \text{if}\ |\lambda|> 1.
    \end{cases}
\end{align}
For small $\l$, this can be made a bit more precise: $|\HCc(\lambda)|^{-1} \le |\l|^d \sigma_1(\l)$ with a smooth function $\sigma_1$ on $\mfa^*$.
The other two properties of $\HCc$ we will use are that for $\lambda\in \mfa^*$ and $s\in \mfw$
\begin{equation}\label{c-conjugate}
|\HCc(s\lambda)|=|\HCc(\lambda)|,\quad \HCc(-\lambda)=\overline{\HCc(\lambda)}.
\end{equation}
The following is also known:
\begin{theorem}[\text{\cite[Proposition 4.7.15]{gangolli1988harmonic}}]\label{estimate of c function}
    The function $\HCc^{-1}$ and its partial derivatives of any order are majorized by 
    \[
    C(1+\abso{\lambda})^{(1/2)(n-l)},
    \]
    $\l\in \mfa^*$.
\end{theorem}

Let us give some definitions concerning standard parabolic subgroups.
Fix a nonzero $H_0\in \Ca$, and let $F\subset \Sigma_s^+$ be the subset of simple roots vanishing at $H_0$. Since $\mfg$ is semisimple, $F$ cannot be equal to $\Sigma_s^+$. Set $\mfa_F = \{ H\in \mfa: \a(H)=0\ \text{for all}\ \a\in F\}$, and let $\mfm_{1F}$ be the centralizer of $\mfa_F$ in $\mfg$. If we denote by $M_{1F}$ the centralizer of $\mfa_F$ in $G$, then $\mfm_{1F}$ is corresponding Lie algebra of $M_{1F}$.
The group $M_{1F}$ is reductive, and $K_F = K \cap M_{1F}$ is its maximal compact subgroup \cite[Theorem 2.3.2]{gangolli1988harmonic}.

Denote by $\Sigma_F^+ \subset \Sigma^+$ the set of positive roots vanishing at $H_0$ and by $\Sigma_F$ the union $\Sigma_F^+\cup -\Sigma_F^+$. Then $F$ and $\Sigma_F^+$ are the simple and positive systems of the pair $(\mfm_{1F}, \mfa)$ respectively, so that the Weyl vector of $M_{1F}$ is $\rho_F:=\frac{1}{2}\sum_{\alpha\in \Sigma_F^+}m_\alpha \alpha$.
We denote by $\mfw_F$ the Weyl group of pair $(\mfm_{1F},\mfa)$  and by $\HCc_F$ the Harish-Chandra $\HCc$-function with respect to $M_{1F}$.
\begin{proposition}\label{cF&c}
     Let $\sigma(\l):=\Bracket{\HCc_F\bar\HCc}^{-1}(\l)$ for any $\l\in\mfa^*$, and $m$ a multi-index, then
    \[
    \abso{\partial^{m} \sigma(\l)} \le C_{l,m} (1+|\l|)^{n-l}
    \]
    where $C_{l,m}$ is a constant independent of $\lambda$.
\end{proposition}
\begin{proof}
    By the Leibniz rule and \Cref{estimate of c function} we conclude that the partial derivatives of $\sigma$ are linear combinations, with constants independent of $\l$, of
    $$
    \partial^{m'} \bar\HCc^{-1} (\l)\cdot \partial^{m''} \HCc_F^{-1} (\l)
    $$
    which are bounded by $C(1+|\l|)^{n-l}$ with a constant $C$ independent of $\lambda$.
\end{proof}
For $|\l|\le 1$, an estimate similar to \eqref{estiofHCc} holds: $|\HCc_F(\lambda)|^{-1} \le |\l|^{d_F}$, where $d_F$ is the number of reduced roots of $\mfm_{1F}$. In general, all one can say is that $d_F \ge|F|$, so we keep the following:
\begin{equation}\label{estiofHCcF}
    |\HCc_F(\lambda)|^{-1} \le |\lambda|^{|F|} \sigma_F(\l)
\end{equation}
with a smooth function $\sigma_F$ on $\mfa^*$.

Let $C(G//K)$(resp. $C_c(G//K)$) denote the space of continuous functions (resp. continuous functions with compact support) on $G$ which are bi-invariant under $K$. 
The spherical transform, called also the Harish-Chandra transform, of $f\in C_c(G//K)$ is defined by
\[
(\mathcal{H}f)(\lambda)=\int_G f(x)\varphi_{-\lambda}(x)d x, \quad \lambda\in  \mfa_{\C}^*.
\]
The Harish-Chandra inversion formula is given by
\[
f(x)=C\int_{ \mfa^*}(\mathcal{H}f)(\lambda)\varphi_\lambda(x) \abso{\HCc (\lambda)}^{-2}d\lambda, \quad x\in G
\]
where $C$ is a constant associated with $G$. 

We conclude this subsection with the following \cite[Theorem 5.9.3 (a)]{gangolli1988harmonic} which describes the asymptotic behavior of $\varphi_\l$. 
We will directly cite the theorem with the refined constant $s$ which is pointed out in \cite[Corollary 3.4]{KuznetsovaSong}.
To state it, we need the following result. It is a proposition since the functions $\psi_\lambda$ on $M_{1F}$ are defined there by another formula \cite[(5.8.23)]{gangolli1988harmonic}; we take it however for a definition.
\begin{Notations}[\text{\cite[Proposition 5.8.6]{gangolli1988harmonic}}]\label{definition of theta}
    For a proper subset $F\subset \Sigma_s^+$, let
    \[
    \theta_\lambda(m)=\int_{K_F}\e^{(i\lambda-\rho_F)(H(mk))}dk
    \]
be the elementary spherical functions on $M_{1F}$ corresponding to $\lambda\in \mfa_{\C}^*$.
Let $\lambda\in \mfa^*$ be regular. Then we set
    \[
    \psi_\lambda(m)=\abso{\mfw_F}^{-1} \sum_{s\in \mfw} (\HCc(s\lambda)/\HCc_F(s\lambda))\theta_{s\lambda}(m)
    \]
    for all $m\in M_{1F}$. 
Denote also $\tau_F(H)=\min_{\alpha\in \Sigma^+\backslash \Sigma_F^+}\abso{\alpha(H)}$ for $H\in \mfa$. 
\end{Notations}
\begin{theorem}[\text{\cite[Theorem 5.9.3 (a)]{gangolli1988harmonic}}]\label{thm593a}
    Suppose that $F\subsetneq \Sigma_s^+$. Then there is a constant $C>0$ such that for all $\lambda\in \mfa^*$ and $a\in \CA$ with $\tau_F(\log a)\ge 1$
    \[
    \abso{\e^{\rho(\log a)}\varphi_\lambda(a)-\e^{\rho_F(\log a)}\psi_\lambda(a)}\le C(1+\abso{\lambda})^{6(n-l)+1} (1+\abso{\log a})^{6(n-l)+1}\e^{-2\tau_F(\log a)}.
    \]
\end{theorem}

\subsection{Distinguished Laplacian}\label{sec-dist-lapl}
The Killing form $B$ is positive definite on $\mfp$ and admits a Riemannian structure on $G/K$ and the Laplace--Beltrami operator $\Delta$. It has also a coordinate description as described below.

The bilinear form $B_\theta: (X,Y)\mapsto -B(X,\theta Y)$ is positive definite on $\mfg$ and the Iwasawa decomposition $\mfg=\mfk\oplus\mfa\oplus\mfn$ is orthogonal with respect to it. We choose and orthonormal basis $(H_1,\dots,H_l,X_1,\dots,X_{\dim \mfg})$ in $\mfg$ so that $(H_1,\dots,H_l)$ is a basis of $\mfa$ and $(X_1,\dots,X_{n-l})$ a basis of $\mfn$. With these viewed as left-invariant vector fields on $\mfg$, we have then \cite{bougerol1983exemples}
$$
\Delta = \sum_{i=1}^{l} H_i^2+2\sum_{j=1}^{n-l} X_j^2+\sum_{i=1}^l H_i \cdot (H_i\delta)(e).
$$
For a function $f$ on $G$, set $\tau (f) = \check f$ and $\check f(s):=f(s^{-1})$. Setting $-\widetilde X = \tau \circ X \circ \tau$ for $X\in\mfg$, we get a right-invariant vector field; the operator $\cal L$ defined by
$$
-\cal L = \sum_{i=1}^{l}\widetilde H_i^2+2\sum_{j=1}^{n-l} \widetilde X_j^2
$$
is a distinguished Laplace operator on $S$ and 
has a special relationship with $\Delta$~\cite{cowling1994spectral} :
\[
\delta^{-1/2}\circ\tau \cal L \tau\circ\delta^{1/2} = -\Delta - \abso{\rho}^2=: \Delta_\rho.
\]
The shifted operator $\Delta_\rho$ has the spectrum $[0, \infty)$. Both $\Delta_\rho$ and $\cal L$ extend to positive and self-adjoint operators on $L^2(S, d_l)$, with the same notations for their extensions.

Let $F$ be a bounded Borel function on $[0, \infty)$. 
We can define bounded operators $F(\cal L)$ and $F(\Delta_\rho)$ on $L^2(S)$ via Borel functional calculus.  
Note that $\cal L$ is right-invariant while $\Delta_\rho$ is left-invariant, we can define $k_\LB$ and $k_\La$ to be the kernels of $F(\Delta_\rho)$ and $F(\cal L)$ respectively, a priori distributions, such that 
\begin{equation}\label{kernel of functions}
    \begin{split}
        F(\Delta_\rho)f=f *k_\LB \quad  
        \text{and}\quad  
        F(\cal L)f=k_\La*f, \quad f\in L^2(S).
    \end{split}
\end{equation}
The connection between $\cal L$ and $\Delta$ implies that
\begin{equation}\label{relation of kernels}
	k_\LB=\delta^{1/2} k_\La,
\end{equation}
which was pointed out by Giulini and Mauceri \cite{giulini1993analysis} (also see \cite{bougerol1983exemples}).
It is known that $k_\LB$ is $K$-biinvariant. Thus, not only $k_\LB$ but $k_\La$ can be regarded as a function on $G$ by the formula above.
By inverse formula for the spherical transform, we obtain the exact formula of the kernel as follows:
\begin{align}\label{integral formula of kernel}
    k_\LB(x)=C\int_{ \mfa^*}F(|\lambda|^2)\varphi_\lambda(x) \abso{\HCc (\lambda)}^{-2}d\lambda, \quad x\in G.
\end{align}

\subsection{The method of stationary phase}\label{sec-stat-phase}

This is a well-known method of study of oscillating integrals of the type 
\[
I(x)=\int_a^b g(t) \e^{ixf(t)} dt
\]
with real  $a, b, x, f(t)$ and complex $g(t)$. We follow the treatment due to \cite{Erdelyi1956} in the exposition of \cite[II.3]{Wong2001}. We will need precise estimates of remainders, for this reason we go into certain details of the proof.

We assume the following conditions:
\begin{enumerate}[(a)]
    \item $f'(t)> 0$ in $(a,b) $ and $f(t)$ has the form of \[
    f(t)=f(a)+(t-a)^p f_1(t),\quad f_1(a)\neq 0 
    \] 
    where $p \ge 1$ and $f_1(t)$ is infinitely differentiable on $[a,b]$.
    \item $g$ is infinitely 
    differentiable on $[a,b]$.
\end{enumerate}

With a new variable $u$ such that 
\begin{itemize}
    \item $u^p=f(t)-f(a),\ u\in [0,B]$;
    \item $B^p=f(b)-f(a) $;
    \item $q(u)=g(t)t'(u),\ u\in [0,B]$, and extended arbitrarily to a compactly supported function in $C^{\infty}([0, \infty))$;
    \item $q_1(v)=v^{\frac{1}{p} - 1}q(v^{\frac{1}{p}}),\ v\in [B^p, \infty)$,
\end{itemize}
the integral takes form
$$
I(x) = e^{ixf(a)} \int_0^B q(u) e^{ix u^p} du = e^{ixf(a)} [I_1(x) - I_2(x)],
$$
where each of the integrals
\begin{align*}
I_1(x) &= \int_0^\infty q(u) e^{ix u^p} du,
\\ I_2(x) &= \int_B^\infty q(u) e^{ix u^p} du = \frac1p \int_{B^p}^\infty q_1(v) e^{ix v} dv
\end{align*}
is treated by partial integration. For $I_2(x)$, this is not a problem, while for $I_1(x)$ the following auxiliary functions are constructed \cite[(3.14-3.16)]{Wong2001}: $k_0(u) = e^{ixu^p}$, and for $n\ge1$
\begin{equation}\label{k_n}
k_n(u) = \frac{ (-1)^n} {(n-1)!} \int_u^{u + \infty e^{i\pi/2p} } (z-u)^{n-1} e^{ix z^p} dz,
\end{equation}
with integration performed over the ray $\arg (z-u) = \pi/2p$. It follows that $k_n'(u) = k_{n-1}(u)$ and
\begin{equation}\label{k_n(0)}
k_n(0) = \frac{ (-1)^n} {(n-1)!\,p } \,\Gamma\Big( \frac np \Big)\, e^{i\pi n/2p} x^{-n/p}.
\end{equation}
We note also that \cite[(3.17)]{Wong2001}
\begin{equation}\label{kn(u)}
| k_n(u)| \le \frac1{ (n-1)!\, p} \, \Gamma \Big( \frac np \Big) \, x^{ - n/p}.
\end{equation}

The asymptotics of $I(x)$ is described now by the following theorem:
\begin{theorem}\label{Stationary phase}
    \begin{equation}
    \begin{split}
        I(x)=\e^{ixf(a)}[I_1(x)-I_2(x)]
    \end{split}
\end{equation}
where for any $N\ge1$, $M\ge 1$
\begin{align}
    I_1(x) &= \sum_{n=0}^{N-1} \frac{1}{n!\,p} \,\Gamma\bracket{\frac{n+1}{p}} q^{(n)}(0) \e^{\frac{i\pi (n+1)}{2p}} x^{-\frac{n+1}{p}}+ R_N^{(1)}(x),\\
    I_2(x) &= \e^{ix(f(b)-f(a))} \sum_{n=0}^{M-1} \frac{1}{p} \,q_1^{(n)}(B^p)\bracket{\frac{i}{x}}^{n+1}+ R_M^{(2)}(x)
\end{align}
and the remainders are
\begin{align}
    \begin{split}\label{remainders-general-th}
        R^{(1)}_{N}(x) &= (-1)^{N+1} \Big[ q^{(N)}(0) k_{N+1}(0)  + \int_0^\infty q^{(N+1)}(u) k_{N+1}(u) du  \Big],
        \\ R_M^{(2)}(x) &= \frac{1}{p}\,\bracket{\frac{i}{x}}^M \int_{B^p}^\infty q_1^{(M)}(v) \e^{ixv} dv,
    \end{split}
\end{align}
estimated as
\begin{align}
    \begin{split}\label{remainders-bounds}
        R^{(1)}_{N}(x) & = O(x^{-(N+1)/p}),
        \\ R_M^{(2)}(x) &= o(x^{-M})
    \end{split}
\end{align}
as $x\to+\infty$.
\end{theorem} 
We cite \eqref{remainders-bounds} for completeness, but we will actually need the exact formulas \eqref{remainders-general-th}.
\section{Oscillating integrals with a radial factor}\label{Oscillating integrals with a radial factor}
This section is devoted to the study of integrals of the type
\[
    I_\Psi(H): = \int_{\mfa^*} \Psi(|\l|) e^{i\l(H) } \sigma(\l) d \l
\]
with a fixed $H\in \mfa$. The convolution kernel of a function $\Psi(\Delta_\rho)$ of the Laplace-Beltrami operator can be approximated by such integrals, with $\sigma(\l)=|\HCc(\l)|^{-2}$. In the present section, we write $I_\Psi$ in polar coordinates and apply the stationary phase method to find the asymptotics for a fixed radius $r$. The main theorem is \Cref{main theorem}. It is stated with vague assumptions on $\Psi$, which are made explicit in \Cref{assump-on-psi-basic}. Moreover, first estimates of remainders are obtained in \Cref{rem-basic-estimate}, and finally, more precise estimates are elaborated in \Cref{subsec-deriv-in-r} to be used in further analysis.

\subsection{Stationary phase analysis for a fixed radius}

This section is devoted essentially to a detailed analysis of radial averages of the Harish-Chandra $\HCc$-function. The group in question is not necessarily semisimple: the results hold for any reductive group. To be precise, we understand by this term a group of class $\mathcal {H}$ as defined in the book of Gangolli and Varadarajan \cite{gangolli1988harmonic}.

We cannot use existing theorems on multi-dimensional stationary phase method since our stationary points would not be isolated. 
We do thus a two-step analysis: first for a fixed $r$ in polar coordinates, and then integrating by $r$, using the radial property of our functions. \Cref{main theorem} below is the first step, independent of the function $\Psi$. The second part cannot be done in a general form; \Cref{Global_estimates} contains analysis of functions related to the wave equation, and other cases can be considered similarly. We need therefore to keep much detail of the functions appearing in the asymptotics, what explains references to multiple formulas in the statement.

In the next theorems, the function $\sigma$ satisfies the same estimates as $|\HCc(\l)|^{-2}$. This is indeed our main case, but it will be also applied in \Cref{Global_estimates} to $\sigma(\l) = |\overline{\HCc(\l)}\HCc_F(\l)|^{-1}$ where $\HCc_F$ is the $\HCc$-function of a parabolic subgroup. Thus, the degree $a$ appearing in the statement can be equal to $2d$ or to $d+d_F$ (see \eqref{estiofHCc} and \eqref{estiofHCcF}, as well as surrounding discussions).

To ensure a more concise statement of the theorems, we assume that $\sigma$ is a fixed function on~$\mfa^*$ satisfying the following:
\begin{assumptions}\label{assumption for sigma}
    Let $\sigma$ be a complex-valued smooth function on $\mfa^*$ satisfying for all $\l\in \mfa^*$
    \begin{enumerate}
    \item given any multi-index $m$, $\abso{\partial^{m} \sigma(\l)} \le C_m (1+|\l|)^{n-l}$;
    \item there exists an integer $a\ge0$ such that $\partial^m \sigma (0)=0$ for every multi-index $m$ with $|m|<a$.
    \end{enumerate}
\end{assumptions}
\begin{remark}
    \leavevmode
    \begin{enumerate}
        \item It follow directly from \Cref{estimate of c function} that $|\HCc(\l)\HCc_F(\l)|^{-1}$ satisfies the first assumption.
        \item Regarding the second assumption, we observe from the Gindikin-\linebreak Karpelevich formula \eqref{gindikin-k} for $\HCc$ that every factor takes finite nonzero values at $\l=0$, except $\Gamma(i\ip{\lambda,\alpha_0})$. The gamma function has a simple pole at zero, thus $\HCc(\l)^{-1} = \prod_{\alpha\in \Sigma_r^+} \ip{ \lambda, \alpha} g(\l)$ with a holomorphic function $g$. A similar expression is valid for $\HCc_F^{-1}$. It is clear now that $|\HCc|^{-2}$ and $\big(\!\HCc_F \overline{\HCc})^{-1}$ satisfy the second assumption, with the corresponding integer $a$ taken to be $2d$ or $d+d_F$ respectively.
    \end{enumerate}
\end{remark}

\begin{theorem}\label{main theorem}
    Let $\Psi:[0,+\infty)\to\C$ be a measurable function on $[0,\infty)$.
    For a unit vector $E\in \Ca$ and $h\ge0$, we consider the integral
    \begin{align}\label{I_Psi-in-the-theorem-initial-form}
    I_\Psi(h): = \int_{\mfa^*} \Psi(|\l|^2) e^{ih\l(E) } \sigma(\l) d \l.
    \end{align}
    Then, assuming all integrals converge,
    \begin{align}\label{I_Psi-in-the-theorem}
    \begin{split}
        I_\Psi(h) &=
        C_l\, h^{(1-l)/2} \int_0^\infty \Psi(r^2)
        \big[ \e^{i\pi(l-1)/4-ihr} \sigma(rE) 
        \\&\quad+ \e^{-i\pi(l-1)/4+ihr} \sigma(-rE)\big] r^{(l-1)/2} dr 
         +\int_0^\infty \Psi(r^2) r^{l-1} R(h,r) dr,
    \end{split}
    \end{align}
    where the remainder $R(h,r)$ is given by \eqref{remainders} and the constant $C_l$ depends on $l$ only.
\end{theorem}
\begin{proof}
First, we change coordinates to simplify the phase. The scalar product given by the Killing form turns $\mfa$ into a Euclidean space of dimension $l$, which can be identified with its dual $\mfa^*$, and with $\R^l$. Let $(e_1,\dots,e_l)$ be the standard basis in $\R^l$, and choose $J$ to be a rotation matrix of size $l$ such that $J^{-1}E = e_1$. Then
\begin{equation}\label{integral-with-lambda1}
    I_\Psi(h) 
    = \int_{\mfa^*} \Psi(|J\l|^2) e^{ih\ip{J\l, E}} \sigma(J\l) d \l 
    = \int_{\mfa^*} \Psi(|\l|^2) e^{i\l_1 h } \sigma(J\l) d \l
\end{equation}
where $\lambda_1$ is the first coordinate of $\lambda$. For convenience, denote $\varsigma(\l):=\sigma(J\l)$; at this point, we note that $\varsigma$ is a smooth function on $\R^l$.

\smallskip
\noindent\textbf{Spherical Coordinates Formula.} 
We pass to polar coordinates:
\begin{equation}\label{polar-nots}
	\begin{split}
		\Theta &:= (\cos\theta_1,\ \sin \theta_1\cos \theta_2,\ \dots, \sin \theta_1\cdots \sin \theta_{l-2}\sin \theta_{l-1}),\\
    \lambda &:=(\lambda_1,\cdots, \lambda_l)=r \Theta, \ r\in[0,\infty),
	\end{split}
\end{equation}
with $0\le \theta_1,\dots, \theta_{l-2}\le \pi,\ 0\le \theta_{l-1}\le 2\pi$ and rewrite the integral as
    $$
    I_\Psi(h) = \int_0^\infty \Psi(r^2) \xi(r,h) dr
    $$
    where
    \begin{align*}
    \xi(r,h)
    &=\int_0^\pi  \dots \int_0^\pi  \int_0^{2\pi} 
    \e^{ihr\cos \theta_1} \varsigma( r\Theta ) J(l,r,\theta) d\theta_{l-1} d \theta_{l-2}\dots  d \theta_1
    \end{align*}
    with the Jacobian $J(l,r,\theta)= \bracket{\sin \theta_{l-2}}\dots  \bracket{\sin \theta_1}^{l-2}  r^{l-1}$.
    
    To make the notations less cumbersome, we introduce the function
\begin{equation}\label{Dr-def}
    D_r(\theta_1)= \int_0^\pi \cdots \int_0^\pi \int_0^{2\pi} \sin \theta_{l-2}\cdots \bracket{\sin \theta_2}^{l-3}\varsigma( r\Theta ) d\theta_{l-1}d \theta_{l-2} \cdots d \theta_2,
\end{equation}
then
\begin{equation}\label{xi(r,h)}
    \begin{split}
        \xi(r,h)=r^{l-1} \int_0^\pi \e^{ihr\cos \theta_1} \bracket{\sin \theta_1}^{l-2} D_r(\theta_1) d\theta_1.
    \end{split}
\end{equation}
We can now use the method of stationary phase to deal with $\xi(h,r)$ since $D_r$ is smooth on $[0, \pi]$.
The critical points of the phase $hr\cos\theta_1$ are both endpoints of our interval, $0$ and $\pi$. To put it into the form of \Cref{sec-stat-phase}, with an increasing phase having a single critical point at the left endpoint, we write, taking the conjugate and changing variables:
\[
r^{1-l} \xi(r, h)= \overline{I_0(h,r)} + I_\pi(h,r),
\]
where
\begin{align*}
    I_0(h,r) &=\int_0^{\pi/2} \e^{i(hr)(-\cos \theta_1)} \bracket{\sin \theta_1}^{l-2} \overline{D_r(\theta_1)} d\theta_1
    \\I_\pi(h,r) &=\int_0^{\pi/2} \e^{i(hr)(-\cos \theta_1)} \bracket{\sin \theta_1}^{l-2} D_r(\pi-\theta_1) d\theta_1.
\end{align*}

\noindent\textbf{Integral $I_0(h,r)$.}
In the notations of \Cref{sec-stat-phase}, we have
\begin{align*}
    f(\theta_1) &= -\cos \theta_1,\\
    g(\theta_1) &= \bracket{\sin \theta_1}^{l-2} \overline{D_r(\theta_1)}.
\end{align*}
The function $f$ is increasing on $(0,\frac{\pi}{2})$ and has its unique critical point at $0$. Its second derivative at $0$ is nonzero, so $p=2$. Other notations are:
\begin{align*}
B^p &= B^2 = f(\pi/2)-f(0)=1; \\
u^2 &= -\cos\theta_1+1,\quad u\in [0,1];\\
\theta_1(u) &= \arccos (1-u^2).
\end{align*}
We calculate also
\begin{align*}
\theta_1'(u) &= 2\bracket{2-u^2}^{-1/2},\\
\sin (\theta_1(u)) &= \sin (\arccos (1-u^2))= u(2-u^2)^{1/2},
\end{align*}
which allows to write
\begin{equation}\label{q(u)}
q(u)=g(\theta_1)\theta_1'(u) = 2 u^{l-2}\bracket{2-u^2}^{(l-3)/2} \overline{D_r(\arccos (1-u^2))}
\end{equation}
for $u\in [0,1]$. We will later extend it to $[B^2,+\infty) = [1,+\infty)$.
The function $q_1(v) = v^{-1/2} q(v^{1/2})$, defined on $[1,+\infty)$, will be discussed later.

Let us now determine the main term of the expansion to obtain. We can easily see that
\[
 q^{(k)}(0)=
\begin{cases}
   0&\ \text{if}\ 0\le k< l-2,\\
   2^{(l-1)/2}(l-2)! \overline{D_r(0)} &\ \text{if}\ k= l-2.
\end{cases}
\]
Moreover, $D_r(0)$ is easy to calculate: if $\theta_1=0$ then $\Theta = (1,0,\dots,0)=e_1$ and 
\begin{equation}
	\begin{split}
		D_r(0)
		&=\varsigma(re_1) \int_0^\pi \cdots \int_0^\pi \int_0^{2\pi} \sin \theta_{l-2}\cdots \bracket{\sin \theta_2}^{l-3} d\theta_{l-1}d \theta_{l-2} \cdots d \theta_2 
        \\&= \varsigma(re_1) S_{l-2}
	\end{split}
\end{equation}
where $S_{l-2}=(2\pi^{(l-1)/2} )/\Gamma((l-1)/2) >0$ is the surface area of the unit sphere in $\R^{l-1}$. We see that in general, this value is nonzero.

Applying \Cref{Stationary phase} with $N=l-1$ and $p=2$, we have, for every $M\ge1$, 
\begin{equation}\label{I_0}
    \begin{split}
        I_0(h,r) 
        &=\e^{-ihr}\Bracket{\frac{1}{(l-2)!2}\Gamma\bracket{(l-1)/2}q^{(l-2)}(0)\e^{\frac{i\pi(l-1)}{4}}(hr)^{\frac{-l+1}{2}}+R_{l-1}^{(1)}(h,r)}\\
        &\qquad -\Bracket{\sum_{n=0}^{M-1} \frac{1}{2} q_1^{(n)}(1)\bracket{\frac{i}{hr}}^{n+1}+ \e^{-ihr}R_M^{(2)}(h,r)}\\
        &= C_l \e^{i\bracket{\pi(l-1)/4-hr}} \overline{\varsigma(re_1)} (rh)^{(1-l)/2}+ \e^{-ihr}R_{l-1}^{(1)}(h,r)\\
        &\qquad -\Bracket{\sum_{m=0}^{M-1} \frac{1}{2} q_1^{(m)}(1)\bracket{\frac{i}{hr}}^{m+1}+ \e^{-ihr}R_M^{(2)}(h,r)},\\
    \end{split}
\end{equation}
with the formulas \eqref{remainders-general-th} for the remainders, in which $x=hr$, $p=2$ and $N=l-1$.

\smallskip
\noindent\textbf{Integral $I_\pi(h,r)$.}
We obtain the same $f, p, u, B$, while the functions $g$ and $q$ are replaced respectively by
\[
\tilde g(\theta_1)= \bracket{\sin \theta_1}^{l-2} D_r(\pi-\theta_1)
\]
and
\begin{equation}\label{tilde-q}
\tilde q(u)
=\tilde g(\theta_1) \theta_1'(u)
= 2 u^{l-2}\bracket{2-u^2}^{(l-3)/2} D_r(\pi - \arccos (1-u^2)).
\end{equation}
We find similarly
\[
\tilde q^{(k)}(0)=
\begin{cases}
   0&\ \text{if}\ 0\le k< l-2,\\
   2^{(l-1)/2}(l-2)! S_{l-2} \,\varsigma(-re_1)&\ \text{if}\ k= l-2,
\end{cases}
\] 
and then obtain
\begin{equation}
    \begin{split}\label{I_pi}
        I_\pi(h,r) 
        &=C_l \e^{i\bracket{\pi(l-1)/4-hr}} \varsigma(-re_1)(rh)^{(1-l)/2} +
        \e^{-ihr}\tilde R_{l-1}^{(1)}(h,r)\\
        &\qquad -\Bracket{\sum_{m=0}^{M-1} \frac{1}{2} \tilde q_1^{(m)}(1)\bracket{\frac{i}{hr}}^{m+1}+ \e^{-ihr}\tilde R_M^{(2)}(h,r)}
        \end{split}
\end{equation}
with any $M\ge1$. The remainders are given again by \eqref{remainders-general-th} with $x=hr$, $p=2$ and $N=l-1$, but $q$ and $q_1$ have to be replaced by $\tilde q$ and $\tilde q_1$ respectively.

Combining \eqref{I_0} and \eqref{I_pi}, we then obtain
\begin{equation}\label{xi-full}
    \begin{split}
        r^{1-l}\xi(h,r)
        &= C_l\,\Bracket{\e^{-i\bracket{\pi(l-1)/4-hr}} \varsigma(re_1) + \e^{i\bracket{\pi(l-1)/4-hr}} \varsigma(-re_1) } (rh)^{(1-l)/2}
        \\&\quad - \frac12 \sum_{m=0}^{M-1} \big[ (-1)^{m+1} \overline{ q_1^{(m)}(1) } + \tilde q_1^{(m)}(1) \big] \bracket{\frac{i}{hr}}^{m+1}
        \\&\quad +\e^{ihr}\overline{ R_{l-1}^{(1)}(h,r)}
        +\e^{-ihr} \tilde R_{l-1}^{(1)}(h,r) 
        \\&\quad - \e^{ihr} \overline{ R_M^{(2)}(h,r)} 
        - \e^{-ihr}\tilde R_M^{(2)}(h,r).
    \end{split}
\end{equation}
\noindent\textbf{Extension of $q$ and $\tilde q$.}
The functions $q$ and $\tilde q$ are initially defined on $[0,1]$ and have to be extended smoothly to $[1,+\infty)$. The main terms and asymptotics at $h\to\infty$ of \eqref{I_0}, \eqref{I_pi} do not depend on the choice of this extension. But we will need exact estimates of remainders and their dependence of $r$, so we fix an explicit formula as follows.

Set
\begin{equation}\label{p(v)}
p(v) = 2(2v-v^2)^{(l-3)/2},
\end{equation}
then $\overline{q(u)} = u\, p(u^2) D_r(\arccos (1-u^2))$ on $[0,1]$. This formula however has sense on $[0,\sqrt2]$ and defines a smooth function on this interval. We choose a function $\psi_0\in C^\infty(\R)$ equal to 1 on $(-\infty,\sqrt{3/2}]$ and to 0 on $[\sqrt{7/4},+\infty)$, and we extend $q$ so as to have
\begin{equation}\label{q-extended}
\overline{q(u)} = \begin{cases} u\, p(u^2) D_r(\arccos (1-u^2)) \,\psi_0(u), & u\in [0,\sqrt2];\\
 0, & u\ge \sqrt2.
\end{cases}
\end{equation}
Similarly, we set
\begin{equation}\label{tilde-q-extended}
\tilde q(u) = \begin{cases} u\, p(u^2) D_r(\pi - \arccos (1-u^2)) \,\psi_0(u), & u\in [0,\sqrt2];\\
 0, & u\ge \sqrt2.
\end{cases}
\end{equation}
This makes $q$ and $\tilde q$ smooth on $[0,+\infty)$.

\smallskip
\noindent\textbf{The terms with $q_1$ and $\tilde q_1$.}
We claim that
\begin{equation}\label{q1-deriv}
(-1)^{m+1}\overline{q_1^{(m)}(1)} + \tilde q_1^{(m)}(1)=0
\end{equation}
for every $m$.
Recall that $q_1,\tilde q_1$ are defined on $[1,+\infty)$ as
$$
 q_1(v)=v^{-1/2 }q(v^{1/2}), \quad \tilde q_1(v) = v^{-1/2 }\tilde q(v^{1/2}).
$$
We can however consider them on $[0,+\infty)$ as given by the same formulas, even if the values on $[0,1]$ do not enter in \eqref{I_0} or \eqref{I_pi}.

Denote $s(x) = D_r(\arccos x)$, then according to \eqref{q-extended} and \eqref{tilde-q-extended} we have for $v\in [0,3/2]$
\begin{align*}
\overline{q_1(v)} &= p(v) s(1-v),\\
\tilde q_1(v) &= p(v) s(v-1),
\end{align*}
since $\pi-\arccos x = \arccos(-x)$ for $x\in[0,1]$.
It is easy to check that $p(2-v)=p(v)$, so we have
\begin{equation}\label{q1-tilde-q1}
\overline{q_1(2-v)} = \tilde q_1(v)
\end{equation}
for $v\in[1/2,3/2]$.
It follows immediately that $\tilde q_1^{(m)} (v) = (-1)^m \overline{q_1^{(m)} (2-v)}$, which proves the formula \eqref{q1-deriv}. 

Returning to the equality \eqref{xi-full}, we have shown that the sum involving $q_1$ and $\tilde q_1$ vanishes.
This observation implies at once that we can rewrite \eqref{xi-full} as
\begin{equation}\label{xi-final}
    \begin{split}
        r^{1-l}\xi(h,r)
        &= C_l\, \big[ \e^{-i\pi(1-l)/4-ihr} \varsigma(re_1) + \e^{i\pi(1-l)/4+ihr} \varsigma(-re_1) \big] (rh)^{(1-l)/2} 
        \\&\quad +R(h,r),
    \end{split}
\end{equation}
giving the main term of \eqref{I_Psi-in-the-theorem}. In
\begin{equation}\label{remainders-final}
R(h,r):=\e^{ihr}\overline{ R_{l-1}^{(1)}(h,r)}+\e^{-ihr} \tilde R_{l-1}^{(1)}(h,r) - \e^{ihr} \overline{ R_M^{(2)}(h,r)}  - \e^{-ihr}\tilde R_M^{(2)}(h,r),
\end{equation}
we can take any $M\ge1$. The remainders are given by the formulas \eqref{remainders-general-th} with $x=hr$, $p=2$ and $N=l-1$, where for $\tilde R^{(1)}_{l-1}$, $\tilde R^{(2)}_M$ the functions $q$ and $q_1$ have to be replaced by $\tilde q$ and $\tilde q_1$ respectively. We can thus rewrite them as follows: 
\begin{align}
\label{remainders}
        R(h,r) &= R^{(0)}(h,r) + R^{(1)}(h,r) + R^{(2)}(h,r),
        \\R^{(0)}(h,r) &= (-1)^{l} \big[ \e^{ihr} \overline{q^{(l-1)}(0) k_{l}(0)} + \e^{-ihr} \tilde q^{(l-1)}(0) k_{l}(0) \big]\notag
        \\&=  \Gamma( l/2 )/(2 (l-1)!)  \,  (hr)^{-l/2}
        \big[ \e^{ihr-i\pi l/4} \overline{q^{(l-1)}(0) } + \e^{-ihr+i\pi l/4} \tilde q^{(l-1)}(0) \big],\notag
        \\ R^{(1)}(h,r) &= (-1)^{l} \int_0^\infty
        \Big[ \e^{ihr} \overline{q^{(l)}(u) k_{l}(u)} + \e^{-ihr} \tilde q^{(l)}(u) k_{l}(u) \Big] du,\notag
        \\ R^{(2)}(h,r) &= \frac{-1}{2} \frac{i^M}{(hr)^M} \int_{1}^\infty 
        \Big[ (-1)^M e^{ ihr(1-v)} \overline{q_1^{(M)}(v)} + e^{ ihr(v-1)} \, \tilde q_1^{(M)}(v) \Big] dv.\notag
\end{align}

\end{proof}

\subsection{Estimates of derivatives in $u$}\label{sec-deriv}

Our next task is to obtain estimates of all functions involved in the remainders, and of remainders themselves, with a control of $r$.
We start with the following technical lemma.
\begin{lemma}\label{derivative of function with J}
    Let $f$ be a smooth function defined on $\R^l$ and $J=(J_{ij})_{1\le i, j\le l}$ a rotation matrix. Denote by $\mu(\l)=(\mu_1(\l), \cdots, \mu_l(\l)):= J\l $ for any $\l\in \R^l$. Then for any multi-index $m$,
    \[
    \partial^m (f\circ J)(\l) = \sum_k \partial^k f(\mu)j_{m, k}
    \]
    where the summation above exhausts the set 
    $$\set{k=(k_1, \cdots, k_l): k_i\ge 0, |k|=|m|}$$ and coefficients $| j_{m, k} | \le l^{|m|-1}$.
\end{lemma}
\begin{proof}
    We shall induct on $|m|$. If $|m|=1$, for any $1\le i\le l$, it is obvious that
    \[
    \frac{\partial (f\circ J)}{\partial \l_i }(\l)= \sum_{j=1}^l  \frac{\partial\, f }{\partial \mu_j }(\mu) \,\frac{\partial \mu_j}{\l_i}(\l) 
    = \sum_{j=1}^l  \frac{\partial\, f }{\partial \mu_j }(\mu) \,J_{ji}.
    \]
    Assuming now the proposition holds for $|m|$, we have for any $1\le i\le l$,
    \begin{align*}
    \frac{\partial }{\partial \l_i} \partial^m(f\circ J) (\l)
    &=\frac{\partial}{\partial \l_i} \sum_{k_1+\dots +k_l=|m| } \frac{\partial^{|m|} f}{\partial \mu_1^{k_1} \dots \partial \mu_l^{k_l }}(\mu) \, j_{m, k_1,\dots, k_l}
    \\& = \sum_{k_1+\dots +k_l=|m| } \sum_{j=1}^l \frac{\partial^{|m|+1} f}{\partial \mu_j\partial \mu_1^{k_1} \dots \partial \mu_l^{k_l }}(\mu)\, J_{ji} j_{m, k_1,\dots, k_l}
    \\&= \sum_{k_1+\dots +k_l=|m| } \sum_{j=1}^l \frac{\partial^{|m|+1} f }{\partial \mu_1^{k_1} \dots  \partial \mu_j^{k_j+1} \dots \partial \mu_l^{k_l }}(\mu)\, J_{ji} j_{m, k_1,\dots, k_l}
    \\&= \sum_{k_1+\dots +k_l=|m|+1 }\frac{\partial^{|m|+1}f }{\partial \mu_1^{k_1} \dots \partial \mu_l^{k_l }}(\mu)\, \sum_{j=1}^l J_{ji} j_{m, k_1,\dots, k_j-1,\dots, k_l}
    \end{align*}
    where coefficients with strictly negative indices are assumed to be zero, this is, $j_{m, k_1,\dots, k_j-1,\dots, k_l}=0$ if $k_j-1< 0$.
    With the direct calculation
    \[
    \abso{\sum_{j=1}^l J_{ji} j_{m, k_1,\dots, k_j-1,\dots, k_l}}
    \le \sum_{j=1}^l \abso{J_{ji}}  \abso{j_{m, k_1,\dots, k_j-1,\dots, k_l}}
    \le \sum_{j=1}^l 1\cdot l^{|m|-1} =l^{|m|}
    \]
    we reach the estimate for the coefficient described in the lemma.
\end{proof}

\begin{remark}\label{assumption of varsigma}
    This implies, in particular, that $\varsigma$ satisfies the same Assumptions \ref{assumption for sigma} as $\sigma$. More precisely, (1) For any multi-index $m$, $\abso{\partial^{m} \varsigma(\l)} \le C_m (1+|\l|)^{n-l}$.
    (2) For any multi-index $m$ with $|m|<a$, $\partial^m\varsigma(0)=0$.
\end{remark}

The analysis of the remainder primarily relies on a detailed study of the functions $q$, $\tilde q$, $q_1$ and $\tilde q_1$, which in our setting are composed of the function $ D_r$ and, more fundamentally, of the function $ \sigma $. We proceed by estimating these functions individually, one by one.

\noindent\textbf{Derivatives of $D_r$.}
In the definition \eqref{Dr-def} of $D_r$ as an integral, only the $\varsigma$ term depends on $\theta_1$, and it is a smooth function, so that we can differentiate by $\theta_1$ under the integral sign:
\begin{equation}\label{derivative of Dr wrt theta}
\frac {d^m} { d\theta_1^m} D_r(\theta_1)= \int_0^\pi\! \cdots \!\int_0^\pi \int_0^{2\pi} {\cal S}\, \frac {d^m} {d\theta_1^m} \varsigma(r\Theta) d\theta_{l-1}d \theta_{l-2} \cdots d \theta_2  
\end{equation}
where for brevity we denote ${\cal S} = \sin \theta_{l-2}\cdots \bracket{\sin \theta_2}^{l-3} $.
The first derivative is
\begin{equation}\label{dsigma-by-theta1}
\frac{d}{d\theta_1}\varsigma(r\Theta)
 = \sum_{i=1}^l \frac{\partial \varsigma}{\partial \l_i}(r\Theta) \frac{d \l_i}{d \theta_1}
 = - \frac{\partial \varsigma}{\partial \l_1}(r\Theta)\, r \sin \theta_1 + \frac{\partial \varsigma}{\partial \l_2}(r\Theta) \,
 r \cos\theta_1 \cos\theta_2 + \dots
\end{equation}
By induction, every further derivative $\dfrac{d^m}{d\theta_1^m}\varsigma(r\Theta)$ is a sum of terms of the type
\begin{align}\label{derivative of varsigma wrt theta}
    \partial^k \varsigma(r\Theta)\, r^{|k|} P_k (\theta_1,\dots, \theta_{l-1})
\end{align}
where $1\le |k|\le m$, and $P_k$ is a trigonometric polynomial not involving $r$.

Based on the observation in \Cref{assumption of varsigma}, we conclude that there is a constant $C_{m,l}$  such that
$$
\Big| \frac {d^m} {d\theta_1^m} \varsigma(r\Theta) \Big| \le C_{l,m} (1+r)^{n-l} \max_{1\le k\le m} r^k
 \le C_{l,m} \, r (1+r)^{n-l+m-1},
$$
and respectively
\begin{align}\label{derivative of D_r}
    \Big| \frac {d^m} { d\theta_1^m} D_r(\theta_1) \Big| \le C_{l, m} \, r (1+r)^{ n-l + m-1}.
\end{align}
Moreover, we have $\dfrac{d^m}{d\theta_1^m}\varsigma(0)=0$ for $m<a$. 

On the other side, every derivative $\dfrac{d^m}{dr^m}\varsigma(r\Theta)$ is a sum of terms of the type
\begin{align}\label{derivative of varsigma wrt r}
    \partial^\mu \varsigma(r\Theta)\, P_\mu (\theta_1,\dots, \theta_{l-1}),
\end{align}
where $1\le |\mu|\le m$, and $P_\mu$ is a trigonometric polynomial not involving $r$. Contrary to \eqref{derivative of varsigma wrt theta}, factors $r^k$ do not appear. It follows from \Cref{assumption of varsigma}(2) that $\left.\dfrac{d^m}{dr^m}\varsigma(r\Theta)\right|_{r=0}=0$ for any $m< a$. 
Suppose now that $r\le 1$. For any $k$ with $|k|< a$ we have, by Taylor's formula of the function $f(r):=\partial^k \varsigma(r\Theta)$ of order $a-|k|-1$ at the point $0$,
\begin{align*}
    |\partial^k \varsigma(r\Theta)| 
    &=\abso{\int_0^r \frac1{(a-|k|-1)!} f^{(a-|k|)}(x) (r-x)^{a-|k|-1} dx}\\ 
    &= r^{a-|k|}\frac1{(a-|k|-1)!} \abso{ \int_0^1 f^{(a-|k|)}(vr) (1-v)^{a-|k|-1} dv }
    \\&\le C_k r^{a-|k|} \int_0^1 \abso{1+vr}^{n-l} (1-v)^{a-|k|-1} dv 
    \\&\le C_{k,l}\, r^{a-|k|},
\end{align*}
and by \eqref{derivative of varsigma wrt theta},
\begin{align}\label{derivative of D_r small r}
    \Big| \frac {d^m} { d\theta_1^m} D_r(\theta_1) \Big| \le C_{l, m} \, r^{ a}, \quad r\le 1.
\end{align}
Note that in both \eqref{derivative of D_r} and \eqref{derivative of D_r small r}, the constants are multiples of $C_{m,\sigma}$ (of \Cref{assumption for sigma}), but we will not keep this index in the sequel.

\noindent\textbf{Derivatives of $q$ and $\tilde q$.}
Denote
$$
s_1(u)=2u^{l-2}(2-u^2)^{(l-3)/2} \psi_0(u),\quad s_2(u)=\overline{D_r(\arccos (1-u^2))},
$$
then by \eqref{q-extended} $q(u)=s_1(u) s_2(u)$ on $[0,\sqrt{7/4}]$ (and 0 otherwise). The function $\psi_0$ is smooth and equal to 1 on $[0,\sqrt{3/2}]$.

For $0\le k\le l$,
\begin{equation}\label{q-k-th-derivative}
q^{(k)}(u) = \sum_{k_1+k_2=k} \frac{k!}{k_1!k_2!} s_1^{(k_1)}(u) s_2^{(k_2)}(u).
\end{equation}
The function $s_1$ is smooth on $[0,\sqrt{7/4}]$, so that 
$$
\max_{ 0\le k_1 \le l} \sup_{u\in [0,\sqrt{7/4}]}\abso{s_1^{(k_1)}(u)} \le C'_l,
$$
with some constant $C'_l$.

To estimate derivatives of $s_2$, for $0\le k\le l$, the $k$-th derivative of $s_2$ by Faa di Bruno's formula \cite[Theorem 2.1]{ConsSa} is a finite sum of terms of the type
\begin{equation}\label{Dr-k-th-derivative}
C_{m_1,\dots, m_k} D_r^{(m_1+\dots +m_k)}\big( \theta_1(u) \big) \prod_{j=1}^k \big[ \theta_1^{(j)}(u) \big]^{m_j},
\end{equation}
with $m_j\ge0$ such that $m_1+\dots+m_k\le k$.

For $u\in[0,\sqrt2)$, $\theta_1'(u) = 2/\sqrt{2-u^2}$, so this and further derivatives of $\theta_1$ are bounded on $[0,\sqrt{7/4}]$, and we conclude that there is a constant $C''_l$ such that
\begin{equation}\label{s2u}
| s_2^{(k)}(u) | \le C''_l \max_{ 0\le m\le k} | D_r^{(m)}\big( \theta_1(u) \big) |
\end{equation}
for all $u\in \supp q$.

Now \eqref{s2u} simplifies to
$$
| s_2^{(k)}(u) | \le C_l (1+r)^{ n-l + k}.
$$
With yet another application of the Leibnitz rule to take into account derivatives of $\psi_0$, we get the final estimate
\begin{equation}\label{q-k-th-derivative-final}
\abso{q^{(k)}(u)}\le C_l (1+r)^{ n-l + k}.
\end{equation}
From \eqref{derivative of D_r small r}, we get similarly for $k\ge1$
\begin{align}\label{derivative of q small r}
    |q^{(k)}(u)| \le C_{l, k} \, r^{ a}, \quad r\le 1.
\end{align}

For the function $\tilde q$, we have similar estimates with
$$
\tilde s_1(u)=2u^{l-2}(2-u^2)^{(l-3)/2},\quad \tilde s_2(u)=D_r(\pi - \arccos (1-u^2)).
$$
The estimates are the same: for $k\ge1$,
\begin{equation}\label{q-tilde-k-th-derivative-final}
\abso{\tilde q^{(k)}(u)}\le C_l (1+r)^{ n-l + k},
\end{equation}
and $|\tilde q^{(k)}(u)| \le C_{l, k} \, r^{ a}$ when $r\le 1$.

\smallskip
\noindent\textbf{Derivatives of $q_1$ and $\tilde q_1$.}
Denote by $\vartheta(v)=\arccos (1-v)$, with \eqref{derivative of D_r}, much as the estimate of $q^{(k)}$, we have 
\begin{align}\label{q1-k-th-derivative-final}
    \abso{q_1^{(k)}(v)}= \abso{\frac{d^k}{d v^k} \Bracket{p(v) D_r(\vartheta(v))}} \le C_{l,k}\, r \,(1+r)^{n-l+k-1},
\end{align}
as well as $|\tilde q_1^{(k)}(v)|\le C_{l,k}\, r \,(1+r)^{ n-l + k-1}$. Moreover,
\begin{equation}\label{q1-k-th-derivative-small-r}
\max\big( |q_1^{(k)}(v)|, |\tilde q_1^{(k)}(v)| \big) \le C_{l, k} \, r^{ a}, \quad r\le1, \ k\ge1.
\end{equation}
As in \eqref{derivative of D_r} and \eqref{derivative of D_r small r}, all these constants are multiples of $C_{k,\sigma}$ of \Cref{assumption for sigma}.

We can now obtain basic estimates of the remainders, though they will be improved later.
\begin{theorem}\label{rem-basic-estimate}
For every $h\ge1$, the remainder $R(h,r)$ in \eqref{remainders} is bounded by
    $$
    \abso{R(h,r)} \le C_l \, (hr)^{-l/2} (1+r)^{n}
    $$
with the choice $M=\lfloor (l+1)/2 \rfloor$.
\end{theorem}
\begin{proof}
By \eqref{remainders}, \eqref{q-k-th-derivative-final} and \eqref{q-tilde-k-th-derivative-final}, knowing that $q$ and $\tilde q$ are supported in $[0,\sqrt2]$, we obtain:
\begin{align*}
    | R^{(0)}(h,r)| & \le C_l (hr)^{-l/2} \big[ |q^{(l-1)}(0)| + |\tilde q^{(l-1)}(0)| \big]
    \le C_l (hr)^{-l/2} (1+r)^{ n-1},
    \\ |R^{(1)}(h,r)| &\le  \int_0^{\sqrt 2} \Big[ | q^{(l)}(u) k_{l}(u)|  +  |\tilde q^{(l)}(u) k_{l}(u)| \Big] du 
    \le C_l (hr)^{ - l/2 } (1+r)^{ n},
\end{align*}
using \eqref{kn(u)} in the final inequality.

Next, \eqref{q1-k-th-derivative-final} implies that
\begin{align}\label{R2-any-M}
\begin{split}
	| R^{(2)}(h,r) | & \le \frac{1}{(hr)^M} \int_{1}^{2} \Big[ |q_1^{(M)}(v)| + |\tilde q_1^{(M)}(v) | \Big] dv
\\&\le C_l (hr)^{-M} r (1+r)^{n-l+M-1}.
\end{split}
\end{align}
If $l$ is odd, we set $M=(l+1)/2$ and get, assuming $h\ge1$,
$$
| R^{(2)}(h,r) | \le C_l (hr)^{-(l+1)/2} r (1+r)^{n-l/2-1/2}
\le C_l (hr)^{-l/2} (1+r)^{n-l/2}.
$$
If $l$ is even, we get the same estimate by choosing $M=l/2$.
\end{proof}

\begin{corollary}\label{assump-on-psi-basic}
 Let $\Psi:[0,+\infty)\to\C$ be a measurable function on $[0,\infty)$. For all integrals in the statement of \Cref{main theorem} to converge, it is sufficient that $M\le l$ and
 \[
   \int_0^\infty |\Psi(r^2)| (1+r)^{n+l/2-2} dr <\infty.
 \]
\end{corollary}
\begin{proof}
The integral \eqref{I_Psi-in-the-theorem-initial-form} itself is bounded by
\begin{equation}\label{Psi-int-rm}
\int_0^\infty |\Psi(r^2)| (1+r)^m dr
\end{equation}
with $m=n-l$, by assumptions on $\sigma$. In the main term of \eqref{I_Psi-in-the-theorem}, the integrals are bounded by similar expressions with $m=n-l+(l-1)/2 = n-l/2-1/2$. \Cref{rem-basic-estimate} implies that the integrals with $ R^{(0)}(h,r)$ and $ R^{(1)}(h,r)$ involve $m = (l-1) + (n -1) -l/2 = n+l/2-2$.

A bit more attention is needed for $R^{(2)}(h,r)$. For $r\ge1$, by \eqref{R2-any-M},
$$
| R^{(2)}(h,r) | \lesssim (1+r)^{-M+1 + n-l+M-1} = (1+r)^{n-l},
$$
so that $\int_1^\infty |\Psi(r^2) R^{(2)}(h,r) | r^{l-1} dr$ is bounded by \eqref{Psi-int-rm} with $m=n-1$.
But for $r\le 1$ we have only $| R^{(2)}(h,r) | r^{l-1} \le r^{l-M}$ (this will be improved in \eqref{R2-derivatives-r}), and in general, we should have $M\le l$. Under this assumption the integral with $\Psi$ over $r\in[0,1]$ is bounded by \eqref{Psi-int-rm} with any positive~$m$.

Finally, $m=n+l/2-2$ is the largest power needed.
\end{proof}

\subsection{Estimates of derivatives in $r$}\label{subsec-deriv-in-r}

We start now preparations for the analysis of the integrals with respect to $r$ appearing in \eqref{I_Psi-in-the-theorem}. This means that we need to estimate derivatives of $r^{l-1} R(h,r)$ rather than of $R(h,r)$ alone.

\smallskip
\noindent\textbf{Derivatives in $r$ of $D_r$.}
By (1) of \Cref{assumption of varsigma} and \eqref{derivative of varsigma wrt r}, we have now
\begin{equation}\label{Dr-derivative-in-r}
\Big| \frac {d^m} {dr^m} D_r (\theta_1) \Big| \le C_{l,m} (1+r)^{n-l}.
\end{equation}

We need next derivatives in $r$ of the derivatives $D_r^{(k)}$ in $\theta_1$. Order of derivation having no influence on the result, we can differentiate $k$ times \eqref{derivative of varsigma wrt r} with respect to $\theta_1$.  The result is the sum of terms
\begin{align}\label{derivative of derivative of varsigma wrt theta1 wrt r}
    \partial^{i+j}\varsigma(r\Theta) r^{|i|} P_{\theta_1,i}(\theta_1,\dots, \theta_{l-1})P_{r,j}(\theta_1,\dots, \theta_{l-1}),
\end{align}
where $1\le |i|\le k$, $|j|=m$, and $P_{\theta_1,i}, P_{r,j}$ are the trigonometric polynomials not involving $r$. This observation gives
\begin{equation}\label{Dr-k-derivative-in-r}
\Big| \frac {d^m} {dr^m} D_r^{(k)} (\theta_1) \Big| \le C_{l,m} (1+r)^{n-l+k}.
\end{equation}
Moreover, an application of Leibnitz's rule yields
\begin{equation}\label{Dr-k-rs-derivative-in-r}
\Big| \frac {d^m} {dr^m} \big( r^s D_r^{(k)} (\theta_1) \big) \Big| \le C_{l,m} (1+r)^{n-l+k+s}
\end{equation}
for any integer $s\ge 0$.

Next, it follows from \Cref{assumption of varsigma} and \eqref{derivative of derivative of varsigma wrt theta1 wrt r} that
\begin{align}\label{derivative of D wrt r and theta value at r=0}
    \dfrac{d^m}{d r^m} D_r^{(k)}(\theta_1)\Big |_{r=0} =0
\end{align}
if $m<a$. Arguing as in the \eqref{derivative of D_r small r}, we consequently obtain
\begin{align}\label{derivative of D_r k rs small r}
        \Big| \frac{d^m}{d r^m} \big( r^{s} D_r^{(k)}(\theta_1) \big) \Big|
        & \le C_{m,s} \, r^{a+s-m}, \qquad r\le1,\ k\ge0.
\end{align}

\smallskip
\noindent\textbf{Derivatives in $r$ of $q$, $\tilde q$, $q_1$ and $\tilde q_1$.}
By \eqref{q-extended}, we have (on its support)
$\overline{q(u)} = u\, p(u^2) D_r(\arccos (1-u^2)) \,\psi_0(u)$, where only $D_r$ depends on $r$. Other functions are smooth and $q$ is compactly supported, thus the previous estimate implies (with possibly another constant)
\begin{equation}\label{q-derivative-in-r}
\Big| \frac {d^m} {dr^m} \big( r^{s} q(u) \big) \Big| \le C_{l,m} (1+r)^{n-l+s}
\end{equation}
for every $u\in [0,+\infty)$.
By \eqref{q-k-th-derivative} and \eqref{Dr-k-th-derivative}, the same reasoning applies to $q^{(k)}$ for any $k$, and we have
\begin{equation}\label{qkth-derivative-in-r}
\Big| \frac {d^m} {dr^m} \big( r^{s} q^{(k)}(u) \big) \Big| \le C_{l,m} (1+r)^{n-l+k+s}
\end{equation}
for $s\ge0$ and any $u$. Moreover,
\begin{equation}\label{qkth-derivative-in-r-small-r}
\Big| \frac{d^m}{d r^m} \big( r^{s} q^{(k)}(u) \big)\Big| \le C_{l,m} \, r^{s+a-m}, \qquad r\le1,
\end{equation}
for every $k\ge0$.
The same holds for the derivatives of $\tilde q$, $q_1$ and $\tilde q_1$. The constants are clearly multiples of $C_{k+m,\sigma}$ of \Cref{assumption for sigma}.

We can formulate this differently as well. By \eqref{derivative of D wrt r and theta value at r=0}, $\dfrac {d^m} {dr^m} q^{(k)}(u)\Big |_{r=0} =0$ for $0\le m\le a-1$. There exists therefore a function $\varkappa_k(u,r)$, smooth in $r$, such that $q^{(k)}(u) = r^a \varkappa_k(u,r)$ (this is best seen from the Taylor's formula with remainder in the integral form). The same holds for $\tilde q^{(k)}$, $q_1^{(k)}$ and $\tilde q_1^{(k)}$.

\smallskip
\noindent\textbf{Derivatives in $r$ of $k_l$.}
By \eqref{k_n} with $n=s$ and $x=hr$, we have for every $s\ge1$
\begin{align*}
\frac{ d^m} {dr^m} k_s(u) &= \frac{ (-1)^s} {(s-1)!} \int_u^{u + \infty e^{i\pi/4} } (z-u)^{s-1} (ihz^2)^m e^{ihr z^2} dz,
\end{align*}
with integration performed over the ray $\arg (z-u) = \pi/4$. Parameterizing $z = u+ \zeta\e^{ i\pi/4}$, we get
\begin{equation}\label{z^2}
z^2 
= u^2 + u\zeta (\sqrt2+i\sqrt2) + i \zeta^2
\end{equation}
and $\Re (ihr z^2) = - hr u\zeta \sqrt2 - hr\zeta^2$. It follows that
\begin{align*}
\Big|\frac{ d^m} {dr^m} k_s(u) \Big| 
&\le C_s \, h^m \int_0^{\infty } \zeta^{s-1} (u+\zeta)^{2m} e^{ - hr\zeta^2} d\zeta
\\& \le C_{s,m} \, h^m \max_{0\le k \le 2m} \int_0^{\infty } \zeta^{s-1+k} e^{ - hr\zeta^2} d\zeta
\\& = C_{s,m}' \, h^m \max_{0\le k \le 2m} (hr)^{-(s+k)/2}
\\& \le C_{s,m}' \, h^m \big[ (hr)^{-s/2} + (hr)^{-s/2 - m} \big]
\\&= C_{s,m}' (hr)^{-s/2} \big[ h^m + r^{-m} \big].
\end{align*}

\begin{theorem}\label{Derivatives_in_r_of_the_remainders}
    For every $h\ge1$ and any $m\ge 1$, the following inequalities
    \begin{equation}\label{R-derivatives-r} \Big| \frac {d^m} {dr^m} \big( r^{l-1} R(h,r) \big) \Big|
    \le C_{l,m} \, h^{m-l/2} (1+r)^{ n + l/2-1}, \quad r\ge 1,
    \end{equation}
    and 
    \begin{equation}\label{R-derivatives-r-small-r}
    \Big| \frac {d^m} {dr^m} \big( r^{l-1} R(h,r) \big) \Big|
    \le C_{l,m} \, h^{m-l/2} \, r^{ a-m + l - \lfloor l/2\rfloor - 2}\, , \quad r\le1,
    \end{equation}
hold with the choice $M=\lfloor l/2 \rfloor +1$.
\end{theorem}
\begin{proof}
We suppose below that $h\ge1$. From \eqref{remainders} and the estimates above we get:
\begin{align}\label{R0-derivatives-r}
&\Big| \frac {d^m } { dr^m} \big( r^{l-1} R^{(0)}(h,r) \big) \Big| 
        \\&= C_l \Big| \frac {d^m } { dr^m} \Big( \frac{ r^{l/2-1} } { h^{l/2} }
        \big[ \e^{ihr-i\pi l/4} \overline{q^{(l-1)}(0) } + \e^{-ihr+i\pi l/4} \tilde q^{(l-1)}(0) \big] \Big) \Big|\notag
       \\ &\le C_{l,m} h^{-l/2} \max_{0\le k\le m} h^{m-k} \notag
       \Big( \Big|\frac {d^k } { dr^k} \big( r^{l/2-1} q^{(l-1)}(0) \big) \Big|
        + \Big|\frac {d^k } { dr^k} \big( r^{l/2-1} \tilde q^{(l-1)}(0) \big) \Big| \Big)\notag
       \\ &\le C_{l,m} \, h^{-l/2} \max_{0\le k\le m} h^{m-k} (1+r)^{n - l +l-1 + l/2 - 1}\notag
       \\ &\le C_{l,m} \, h^{m-l/2} (1+r)^{n + l/2 -2}.\notag
\end{align}
and for $r\le1$, $m\ge1$
\begin{align}
    \begin{split}\label{R0-derivatives-r-small-r}
        \Big| \frac {d^m } { dr^m} \big( r^{l-1} R^{(0)}(h,r) \big) \Big| 
        &\le C_{l,m} h^{-l/2} \max_{0\le k\le m} h^{m-k} \, r^{(l/2-1)+a -k }
        \\& \le C_{l,m} h^{m-l/2} \, r^{a - m + l/2-1}.
    \end{split}
\end{align}
Next, for $r\ge1$
\begin{align}\label{R1-derivatives-r}
        &\Big| \frac {d^m } { dr^m} \big( r^{l-1}  R^{(1)}(h,r) \big) \Big|
        \\&= \Big|\frac {d^m } { dr^m} r^{l-1} \int_0^\infty
        \Big[ \e^{ihr} \overline{q^{(l)}(u) k_{l}(u)} + \e^{-ihr} \tilde q^{(l)}(u) k_{l}(u) \Big] du \Big|\notag
        \\ &\le C_{l,m} \int_0^\infty
        \max_{0\le k\le m} h^{m-k} \Big[
        \Big|\frac {d^k } { dr^k} \big[ r^{l-1} \overline{q^{(l)}(u) k_{l}(u)} \big]\Big|
        + \Big|\frac {d^k } { dr^k} \big[r^{l-1} \tilde q^{(l)}(u) k_{l}(u)\big]\Big| \Big] du\notag
        \\ &\le C_{l,m} \max_{0\le k\le m} h^{m-k} 
        (1+r)^{n-l+l+(l-1)} \max_{0\le j\le k} (hr)^{-l/2} \big[ h^j + r^{-j} \big]\notag
        \\ &\le C_{l,m} \max_{0\le k\le m} h^{m-k} 
        (1+r)^{n+l-1} (hr)^{-l/2} h^k\notag
        \\ &\le C_{l,m} h^{m-l/2} (1+r)^{n+l/2-1},\notag
\end{align}
while for $r\le1$
\begin{align}
    \begin{split}\label{R1-derivatives-r-small-r}
        &\Big| \frac {d^m } { dr^m} \big( r^{l-1} R^{(1)}(h,r) \big) \Big|
        \\&\le C_{l,m} \max_{0\le k\le m} h^{m-k} 
        \max_{0\le j\le k} r^{(l-1)+a-(k-j)} (hr)^{-l/2} \big[ h^j + r^{-j} \big]
       \\ &\le C_{l,m} r^{a+l/2-1}\, h^{m-l/2} \max_{0\le j\le k\le m} \big[ h^{j-k} r^{j-k} + r^{-k}h^{-k} \big]
       \\ &\le C_{l,m} r^{a+l/2-1-m}\, h^{m-l/2}.
    \end{split}
\end{align}
Finally, for any $r$
\begin{align}
    \begin{split}\label{R2-derivatives-r}
        &\Big| \frac {d^m } { dr^m} \big( r^{l-1} R^{(2)}(h,r)  \big) \Big|
        \\&\le h^{-M} \int_{1}^\infty 
        \Big| \frac {d^m } { dr^m} \big( r^{l-1-M} \Big[ e^{ihr(1-v)} \overline{q_1^{(M)}(v)} + e^{ihr(v-1)} \, \tilde q_1^{(M)}(v) \Big] \Big| dv
        \\&\le h^{-M} C_{l,m} \max_{0\le k\le m} h^{m-k} (1+r)^{n-l+M+(l-1-M) }
        \\&\le C_{l,m} h^{m-M} (1+r)^{n-1}
    \end{split}
\end{align}
and for $r\le1$
\begin{align}
    \begin{split}\label{R2-derivatives-r-small-r}
        \Big| \frac {d^m } { dr^m} \big( r^{l-1} R_M^{(2)}(h,r)  \big) \Big|
        &\le h^{-M} C_{l,m} \max_{0\le k\le m} h^{m-k} r^{(l-1-M)+a-k }
        \\&\le C_{l,m} h^{m-M} r^{l-1+a - M - m }.
    \end{split}
\end{align}
Altogether, choosing $M=\lfloor l/2\rfloor +1$, we have \eqref{R-derivatives-r} and for $r\le1$
\begin{equation}
\Big| \frac {d^m} {dr^m} \big( r^{l-1} R(h,r) \big) \Big|
\le C_{l,m} \, h^{m-l/2} \big(r^{ a-m + l/2-1 } + r^{ a+l-1-m-M} \big). 
\end{equation}
Note that $l-1-M$ is equal to $l/2-2$ or $l/2-3/2$ for $l$ even or odd, respectively; this term is therefore dominant at $r\to0$. Thus, for $m\ge1$, we reach \eqref{R-derivatives-r-small-r}.
\end{proof}

\section{Global estimates for oscillatory functions of the shifted Laplacian}\label{Global_estimates}

In this section we study the operator $\Delta_\rho = -\Delta -|\rho|^2$ and functions of it of the form $\Psi(\Delta_\rho)$ with $\Psi(r^2) = e^{itr} \psi(r)$.
By \eqref{integral formula of kernel}, the convolution kernel $k_t$ of $\Psi(\Delta_\rho)$ is
\begin{equation}\label{k_t-LB}
k_t(x)= C_G \int_{ \mfa^*} e^{it|\l|} \psi(|\lambda|)\varphi_\lambda(x) \abso{\HCc (\lambda)}^{-2}d\lambda, \quad x\in G.
\end{equation}
From \eqref{relation of kernels}, we get the corresponding result for the distinguished Laplacian $\cal L$.

\begin{theorem}\label{eit-psi-LB-no-gap}
Let $\psi$ be $p$ times continuously differentiable on $[0,+\infty)$, and $b\ge0$ an integer such that either $b=0$ or $b>0$ and $\psi^{(m)}(0)=0$ for $0\le m< b\le p-2$. Suppose moreover that
\begin{equation}\label{psi-r-integrals}
C_{\psi,k,s} = \int_0^\infty |\psi^{(k)}(r)| (r+1)^{s} dr < \infty
\end{equation}
for $0\le k\le p$ and $s\ge 0$ specified below.
Then we have:
\begin{enumerate}
\item Suppose that $p=\nu$ and \eqref{psi-r-integrals} holds for $s=n-1$. Then for every $x\in G$ and $t\ne0$,
$$
|k_t(x)| \lesssim |t|^{-\nu} (1+|x^+|)^{\nu}\, \phi_0(\exp x^+)
\lesssim |t|^{-\nu} ( 1+ |x^+| ) ^{ \nu+d} e^{-\rho( x^+) }.
$$
\item Suppose that $p=\lfloor l/2 \rfloor$ and \eqref{psi-r-integrals} holds for $s=7(n-l)+1$.\\
For $|t|$ large enough and $|x^+| \le 3|t|$,
$$
|k_t(x)| \lesssim\, |t|^{(1-l)/2} (\log |t|)^d e^{-\rho(x^+)}.
$$
\item Suppose that $p= 2d + \lfloor l/2\rfloor$ and \eqref{psi-r-integrals} holds with $s = \max(7(n-l)+1,n+2d+(l+\lfloor l/2 \rfloor-3)/2)$.
\\
Then for $|t|$ large enough and $|x^+| > 3|t|$,
\begin{align*}
| k_t(x) | &\lesssim |x^+|^{-d-l-b} 
e^{-\rho(x^+)}.
\end{align*}
\item
There is a constant $\delta$ depending on the group only such that for $|t|$ large enough and $x$ with $\log |t|<|x^+|\le 3|t| $ and $\big| |t|-|x^+| \big| > \delta \log |t|$
\begin{align*}
| k_t(x) | \lesssim \big( |x^+|^{-d-b-l} &+ |x^+|^{(1-l)/2} \big| |x^+| - |t| \big|^{-d-b-(l+1)/2} \big) 
   \\&\times(\log |t|)^{2(2d+b+l)+1} e^{-\rho(x^+)}.
\end{align*}
\end{enumerate}
In each case, the constant is $C_G \max\limits_{0\le k\le p} C_{\psi,k,s}$, with $C_G$ depending on the group.
\end{theorem}

\begin{corollary}\label{eit-psi-L-no-gap}
Let $\psi$ satisfy the assumptions of Theorem \ref{eit-psi-LB-no-gap}, and let $k_{t,\cal L}$ be the kernel of $\Psi(\cal L)$ with $\Psi(r^2) = \e^{itr} \psi(r)$. Then:
\begin{enumerate}
\item Suppose that $p=\nu$ and \eqref{psi-r-integrals} holds for $s=n-1$. Then for every $x\in G$ and $t\ne0$,
$$
|k_{t,\cal L}(x)| \lesssim |t|^{-\nu} ( 1+ |x^+| ) ^{ \nu+d}.
$$
\item Suppose that $p=\lfloor l/2 \rfloor$ and \eqref{psi-r-integrals} holds for $s=7(n-l)+1$.\\
For $t$ large enough and $|x^+| \le 3t$,
$$
|k_{t,\cal L}(x)| \lesssim\, t^{(1-l)/2} (\log t)^d.
$$
\item Suppose that $p= 2d + \lfloor l/2\rfloor$ and \eqref{psi-r-integrals} holds with $s = \max(7(n-l)+1,n+2d+(l+\lfloor l/2 \rfloor-3)/2)$.
\\
Then for $t$ large enough and $|x^+| > 3t$
\begin{align*}
| k_{t,\cal L}(x) | &\lesssim |x^+|^{-d-l-b}.
\end{align*}
\item
There is a constant $\delta$ depending on the group only such that for $t$ large enough and $x$ with $\log t<|x^+|\le 3t $ and $|t-|x^+| \,| > \delta \log t$
\begin{align*}
| k_{t,\cal L}(x) | \lesssim \big( |x^+|^{-d-b-l} & + |x^+|^{(1-l)/2} \big| |x^+| - |t| \big|^{-d-b-(l+1)/2} \big) 
   \\&\times(\log |t|)^{2(2d+b+l)+1} .
\end{align*}
\end{enumerate}
In each case, the constant is $C_G \max\limits_{0\le k\le p} C_{\psi,k,s}$, with $C_G$ depending on the group.
\end{corollary}

\begin{remark}
In these theorems, the strongest differentiability assumption is $p=\nu$ in (1), and the weakest one is $p=\lfloor l/2 \rfloor$ in (2).
\end{remark}

\begin{remark}\label{minimal_alpha}
    When $\psi(x)=(1+x^2)^{-\alpha/2}$, the boundedness of the operators $\Psi(\Delta)$ or $\Psi(\mathcal L)$ allows one to derive estimates for solutions to the wave equation in terms of Sobolev norms of the initial data. Therefore, it is important to determine the conditions on $\alpha$ under which the corresponding operators are bounded. In the above theorem concerning kernel estimates,
    \begin{itemize}
        \item for Statement~1, condition \eqref{psi-r-integrals} is satisfied whenever $\Re\alpha > n$;
        \item for Statements~2--4, condition \eqref{psi-r-integrals} is satisfied whenever
        \[
            \Re\alpha > \max\!\left(7(n-l)+2,\; n+2d+(l+\lfloor l/2 \rfloor -1)/2 \right).
        \]
    \end{itemize}
\end{remark}

\begin{remark}\label{case1,m<=nu}
Cases (1) of Theorem \ref{eit-psi-LB-no-gap} and Corollary \ref{eit-psi-L-no-gap} can be also stated as follows:
Suppose that $p\le\nu$ and \eqref{psi-r-integrals} holds for $s=n-1$. Then for every $x\in G$, $m\le p$ and $t\ne0$,
$$
|k_t(x)| = e^{-\rho( x^+) } |k_{t,\cal L}(x)| \lesssim |t|^{-m} (1+|x^+|)^{m}\, \phi_0(\exp x^+).
$$
Up to a factor depending on the group, the constant is the maximum of $|\psi^{(k)}(0)|$ with $0\le k\le m-1$ and of the integrals \eqref{psi-r-integrals} with $s=n-1$ and $0\le k\le m$. This will be proved together with (1) of the theorem.
\end{remark}

\subsection{Some Lemmas}
In the proof we will apply the stationary phase method to several integrals of the following type:
\begin{lemma}\label{stphase-oscillating-no-gap}
Let $\sigma$ satisfy \Cref{assumption for sigma}, and $\psi$ those of \Cref{eit-psi-LB-no-gap} with $s =  n+l/2-1$.
Then for $H\in\mfa$ with $h := \|H\|\ge1$ the integral
\begin{equation}\label{I_t-lem1-no-gap}
I_t(H) = \int_{\mfa^*} e^{it|\l|} \psi(|\l|) \e^{i\lambda(H)} \sigma (\lambda) d\lambda
\end{equation}
is bounded as
\begin{equation}\label{ktL-1}
| I_t(H) | \le C_{m,\psi} |t|^{-m} h^{m+(1-l)/2}
\end{equation}
for every $m\le \min(p, a + l - \lfloor l/2\rfloor - 1)$.
\end{lemma}
\begin{proof}
Denote $e_H = H / h$.
By \Cref{main theorem}, $I_t(H)$ is equal to
    \begin{equation}\label{stat-sec4}
    \begin{split}
        C_l\, h^{(1-l)/2} \int_0^\infty e^{itr} f_H(r) dr
        +\int_0^\infty e^{itr} \psi(r) r^{l-1} R(h,r) dr,
    \end{split}
    \end{equation}
    where $f_H(r) = \psi(r) r^{(l-1)/2} \big[ \e^{i\pi(l-1)/4-ih r} \sigma (re_H) + \e^{-i\pi(l-1)/4+ih r} \sigma (-re_H) \big]$ and the derivatives of the remainder $R(h,r)$ are bounded by \eqref{R-derivatives-r} and \eqref{R-derivatives-r-small-r}.

By assumption, we have $f^{(k)}_H(0) = 0 $ for $k <a+ (l-1)/2$, and can integrate by parts: for $m\le \min(p, a+(l+1)/2)$
\begin{align}\label{int-f_H-lem}
    \int_0^\infty e^{itr} f_H(r) dr
    & = \frac1{(-it)^{m-1}} \Big[ -\frac1{it} e^{itr} f_H^{(m-1)}(0) - \frac1{it}\int_0^\infty e^{itr} f_H^{(m)}(r) dr\Big],
\end{align}
which is bounded by $C_{m,\psi} h^{m}  |t|^{-m}$, estimating the first term in \eqref{stat-sec4} by 
\[
C_{m,\psi} h^{m+(1-l)/2}  |t|^{-m}.
\]
The constant $C_{m,\psi}$ is bounded, as it is easy to see, by the maximum of integrals \eqref{psi-r-integrals} with $k\le m$ and $s\le (n-l)+(l-1)/2 = n-(l+1)/2$ (multiplied by a constant depending on $l$ only).

For the remainder, by \eqref{R-derivatives-r-small-r} we have $r^{l-1} R(h,r) = O( r^{a + l - \lfloor l/2\rfloor - 2} )$ at $r\to0$, thus we can integrate by parts in the smaller range $m\le \min(p, a + l - \lfloor l/2\rfloor - 1)$ arriving, by \eqref{R-derivatives-r}, at
$$
C_{m,\psi} \, h^{m-l/2} |t|^{-m}.
$$
which has a smaller power of $h$ than the first estimate, as $h\ge1$.
The constant however involves integrals \eqref{psi-r-integrals} with $s \le  n+l/2-1$.
\end{proof}

\begin{lemma}\label{no-gap-int-of-rs-qk}
Let $\sigma$ be as defined in \Cref{assumption for sigma}, and $\psi$ those of \Cref{eit-psi-LB-no-gap} with $s=n-l+b+k$, for some $b,k\in\N$ ($k\ge1$). Then for $H\in \mfa$ and every $v\in\R$, the integral
$$
I = \int_0^\infty e^{ivr} \psi(r)\, r^b \overline{q^{(k)}(0)} dr
$$
is bounded as $|I|\le C_{m,\psi} \, |v|^{-m}$ for every $m\le \min(p,a+b+1)$.
The same estimate is valid for the integral with $\bar q$ replaced by $\tilde q$, $q_1$ or $\tilde q_1$.
\end{lemma}
\begin{proof}
Denote $f(r) = \psi(r)\, r^b \overline{q^{(k)}(0)}$. By \eqref{qkth-derivative-in-r}, every derivative of $r^b \overline{q^{(k)}(0)}$ is bounded by $C (1+r)^{n-l+b+k}$. The Leibnitz' rule then implies that
\begin{equation}\label{lem-powers-of-psi-and-r}
\big| f^{(m)}(r) \big| \le
 C (1+r)^{n-l+b+k} \max_{0\le j\le m} |\psi^{(j)}(r)|
\end{equation}
for every $m\le p$.
By \eqref{qkth-derivative-in-r-small-r}, the function $f$ has a zero of order $b+a$ at $r=0$, and moreover, it is compactly supported. Integrating by parts, we get for $m\le \min(p,b+a+1)$:
\begin{align}\label{I-integration by parts}
        |I| 
& = \Big| \frac{(-1)^m }{(iv)^m} f^{(m-1)} (0) + \frac{(-1)^m}{(iv)^m} \int_0^\infty e^{ivr} f^{(m)}(r) dr \Big|
        \le C_{m,\psi} |v|^{-m}.
\end{align}
The calculations for $\tilde q$, $q_1$ and $\tilde q_1$ are identical.
\end{proof}

\begin{lemma}\label{int-with-r1/2}
Let $f$ be a function on $[0,+\infty)$, continuously differentiable $m+1$ times, with an integer $m\ge0$, and having a zero of order $\ge m$ at $0$. Suppose also that
\begin{equation}\label{M-in-lemma}
M = |f^{(m)}(0)| + \int_0^\infty |f^{(m+1)}(r)| dr < \infty.
\end{equation}
Then there exists a constant $C>0$ depending on $m$ only such that for any nonzero $v\in \R$ 
\begin{equation}\label{int-with-r-1/2}
\Big| \int_0^\infty e^{ivr} r^{-1/2} f(r) dr \Big| \le C M |v|^{-m-1/2}.
\end{equation}
\end{lemma}
\begin{proof}
We have by assumption $f^{(k)}(0)=0$ for $0\le k\le m-1$. By Taylor's formula of order $m-k-1$ for $f^{(k)}$, for every $0\le k\le m-1$ we can write
$$
f^{(k)}(r) 
= r^{m-k} \int_0^1 f^{(m)}(vr) \frac{ (1-v)^{m-k-1} } { (m-k-1)! } dv =: r^{m-k} g_k(r),
$$
and note that $g_k\in C^1([0,+\infty))$. 

Denote $f_1(r) = r^{-1/2} f(r)$. For every $0\le k\le m$, we have
$$
f_1^{(k)}(r) = r^{-1/2} \sum_{j=0}^k c_{jk}\, r^{j-k} f^{(j)}(r)
$$
with scalar coefficients $c_{jk}$; for $k=0$ this is obvious, and for $k>0$ proved by induction.

If we set $g_m=f^{(m)}$ and $c_{mm}=1$, we have then for $0\le k\le m$
\begin{equation}\label{f1k}
f_1^{(k)}(r) =
r^{-1/2} \sum_{j=0}^k c_{jk}\, r^{m-k} g_j(r),
\end{equation}
and $f_1^{(k)}(0)=0$ if $0\le k\le m-1$. Note that $f_1^{(m)}$ is integrable on $[0,+\infty)$.
Integrating by parts $m$ times, we obtain similarly to \eqref{int-f_H-lem}
\begin{align*}
\int_0^\infty e^{ivr} f_1(r) dr
 &= \frac1{(-iv)^{m-1}} \Big[ -\frac1{iv} e^{ivr} f_1^{(m-1)}(0) - \frac1{iv}\int_0^\infty e^{ivr} f_1^{(m)}(r) dr\Big]
\\ & = \frac1{(-iv)^{m}} \int_0^\infty e^{ivr} f_1^{(m)}(r) dr.
\end{align*}

Denote $f_2(r) = r^{1/2} f_1^{(m)}(r)$. The change of variable $r=u^2$ takes the integral to a standard form:
\begin{align*}
I: = \frac12 \int_0^\infty e^{ivr} r^{-1/2} f_2(r) dr
& = \int_0^\infty e^{ivu^2} f_2(u^2) du.
\end{align*}
Using the functions $k_n$ \eqref{k_n} and their estimates \eqref{kn(u)} (in these formulas $v$ needs to be positive, so we conjugate $I$), we can integrate by parts
$$
I = \int_0^\infty k_0(u) f_2(u^2) du = - k_1(0) f_2(0) - 2\int_0^\infty k_1(u) f'_2(u^2) u du
$$
and estimate
$$
|I| \le C |v|^{-1/2} \big( |f_2(0)| + 2\int_0^\infty |f'_2(u^2) u| du\big).
$$
It follows from \eqref{f1k} that 
$|f_2(0)| \le C_m |f^{(m)}(0)|$. Next,
$$
f_2'(r) =  \sum_{k=0}^m c_{km}\, g_k'(r) = \sum_{k=0}^{m-1} c_{km}\, \int_0^1 v f^{(m+1)}(vr) \frac{ (1-v)^{m-k-1} } { (m-k-1)! } dv +f^{(m+1)}(r),
$$
so that
\begin{align*}
2\int_0^\infty |f'_2(u^2) u| du
& = \int_0^\infty |f_2'(r)|dr
\\& \le  \sum_{k=0}^m c'_{km}\, \int_0^\infty \int_0^1 v | f^{(m+1)}(vr)|  dv dr + \int_0^\infty \abso{f^{(m+1)}(r)} dr
\\&  = C_m \int_0^\infty | f^{(m+1)}(x)| dx .
\end{align*}
\end{proof}

\begin{corollary}\label{no-gap-int-of-rs-qk-1/2}
Let $j\ge 1$ and $c\ge1/2$ a half-integer. Let $\sigma$ be as defined in \Cref{assumption for sigma}, and $\psi$ those of \Cref{eit-psi-LB-no-gap} with $s=n-l+c+j$. Then for $H\in \mfa$ and every $v\in\R$, the integral
$$
I = \int_0^\infty e^{ivr} \psi(r)\, r^c \overline{q^{(j)}(0)} dr
$$
is bounded as
$$
|I|\le C \, \begin{cases} |v|^{-(a+b+c+1)},  &p \ge a+b+\lceil c\rceil +1;
\\ |v|^{-p+1/2}, & p<a+b+\lceil c\rceil +1,
\end{cases}
$$
where the constant $C$ is the maximum of $C_{\psi,k,s}$ over $0\le k\le a+b+\lceil c\rceil +1$ or $0\le k\le p$ respectively, and $s=n-l+j+\lceil c\rceil $.
The same estimate is valid for the integral with $\bar q$ replaced by $\tilde q$, $q_1$ or $\tilde q_1$.
\end{corollary}
\begin{proof}
It is enough to prove the statement for $q$.
We have $c=c_1-1/2$ with an integer $c_1 = \lceil c\rceil \ge1$, thus we are estimating an integral as in \eqref{int-with-r-1/2} with $f(r) = \psi(r)\, r^{c_1} \overline{q^{(j)}(0)}$.
If $p\ge a+b+c_1+1$, we apply \Cref{int-with-r1/2} with $m=a+b+c_1$, otherwise we set $m=p-1$. The lemma implies then $|I|\le C_m M |v|^{-a-b-c_1-1/2} = C_m M |v|^{-a-b-c-1}$ or $|I|\le C_m M |v|^{-p+1/2}$ respectively, with $M$ given by \eqref{M-in-lemma}.
We have, by \eqref{qkth-derivative-in-r},
$$
| f^{(m+1)}(r) | \le \sum_{k=0}^{m+1} \binom{m+1}k | \psi^{(k)}(r)| (1+r)^{n-l+j+c_1}
$$
which completes the proof.
\end{proof}

\begin{lemma}\label{big-h-no-gap}
Let $\sigma$ be as defined in \Cref{assumption for sigma}, and $\psi$ those of \Cref{eit-psi-LB-no-gap} with $p\ge a + b+\lfloor l/2\rfloor+1$ and $s = n+l/2-1 + \min(p, a + \lfloor l/2 \rfloor+1) /2$.  
Then for $|t|\ge 1$ and $H\in \mfa$ with $h = \|H\|\ge 1$ such that $\big |h-|t| \big|\ge1$, the integral \eqref{I_t-lem1-no-gap} is bounded by
\begin{equation}\label{eq-big-h-no-gap}
C\, h^{(1-l)/2} \min \set{h, \big|h-|t| \big|}^{-a-b-(l+1)/2} \log |t|
\end{equation} 
if $h\le 2|t|$, and by
\begin{equation}\label{eq-big-h-no-gap-exact}
C \, h^{-(a+b+l)}
\end{equation}
if $h>2|t|$.
The constants depend on $\psi$, $p$, on the group, and are multiples of $C_{2(a+b+l),\sigma}$ of \Cref{assumption for sigma}.
\end{lemma}
\begin{proof}
If $t<0$, we replace $I_t(H)$ by its conjugate
\begin{align*}
\overline{I_t(H)} &= \int_{\mfa^*} e^{-it|\l|} \bar\psi(|\l|) \e^{-i\lambda(H)} \bar\sigma (\lambda) d\lambda
\\& = \int_{\mfa^*} e^{-it|\l|} \bar\psi(|\l|) \e^{i\lambda(H)} \bar\sigma (-\lambda) d\lambda
\end{align*}
where $-t>0$, and $\bar\psi$, $\bar\sigma (-\cdot)$ satisfy the same assumptions as $\psi$ and $\sigma$. This allows us to assume below that $t>0$.

\textbf{Step 1.} We write the equality \eqref{stat-sec4} differently:\\
Denote $f_H(r) = \psi(r) \sigma (re_H) r^{(l-1)/2}$ and $\tilde f_H(r) = \psi(r) \sigma (-re_H) r^{(l-1)/2}$, then the main term of $I_t(H)$ is
\begin{align}\label{stat-lem2-sec4-main}
C_l\, h^{(1-l)/2} & \Big[ \e^{-i\pi(l-1)/4} \int_0^\infty \tilde f_H(r) \e^{i (t+h) r} dr 
+ \e^{i\pi(l-1)/4} \int_0^\infty f_H(r) \e^{i(t-h)r} dr \Big].
\end{align}
As in \Cref{stphase-oscillating-no-gap}, we can integrate by parts $m$ times as soon as $m\le \min(p, a+b+(l+1)/2)$.
This time $f_H$ and $\tilde f_H$ do not depend on $h$ (but only on $e_H$), and for any $r$
\begin{align*}
	\Big| \frac {d^m} {dr^m} f_H(r) \Big| 
	&\le C_{m,\psi} C_{m,\sigma} \max_{0\le i\le m} \psi^{(i)}(r) (1+r)^{n-l+ (l-1)/2}
	\\&= C_{m,\psi} C_{m,\sigma} \max_{0\le i\le m} \psi^{(i)}(r) (1+r)^{n-(l+1)/2},
\end{align*}
the bounds for $\tilde f_H$ being identical. We will omit constants depending on $\sigma$ in the proof but keep in mind that every $C_{m,\psi}$ is multiplied by $C_{m,\sigma}$ with the same $m$.
Thus, \eqref{stat-lem2-sec4-main} is bounded, for $h\ne t$, by
$$
C_{m,\psi}\, h^{(1-l)/2} \big[ (h+t)^{-m} + |h-t|^{-m} \big]
\le C_{m,\psi}\, h^{(1-l)/2}|h-t|^{-m},
$$
and we need $s\ge n-(l+1)/2$ in \eqref{psi-r-integrals}. If $l$ is odd, we have by assumption $p\ge a+b+(l+1)/2$, thus we get the bound $C_{\psi}\, h^{(1-l)/2} |h-t|^{-a-b-(l+1)/2}$.

If $l$ is even, we need to improve this a bit. 
We can write $\sigma(re_H) = r^a \sigma_1(r)$ with $\sigma_1$ smooth.
Lemma \ref{int-with-r1/2} 
applies, with $m=a+b+l/2$, to the integrals with both $f_H$ and $\tilde f_H$;
this allows to estimate \eqref{stat-lem2-sec4-main}, for $h\ne t$, by the same power
\begin{equation}\label{lemma-main}
\begin{split}
	&C_{\psi}\, h^{(1-l)/2} \big[ (h+t)^{-(a+b+l/2)-1/2} + |h-t|^{-(a+b+l/2)-1/2} \big]
\\&\le C_{\psi}\, h^{(1-l)/2} |h-t|^{-a-b-(l+1)/2}.
\end{split}
\end{equation}

\textbf{Step 2.} To get estimates of the remainder, we need to return to the formulas \eqref{remainders}.
First, for $R^{(0)}(h,r)$ we have
\begin{align*}
        I^{(0)} 
        &= \int_0^\infty e^{itr} \psi(r)\, r^{l-1} R^{(0)}(h,r) dr
        \\&= C_l h^{-l/2} \Big[ \e^{-i\pi l/4} \int_0^\infty e^{i(t+h)r} \psi(r)\, r^{l/2-1} \overline{q^{(l-1)}(0) } dr
        \\& \quad + \e^{i\pi l/4}  \int_0^\infty e^{i(t-h)r} \psi(r)\, r^{l/2-1} \tilde q^{(l-1)}(0) dr \Big].
\end{align*}
By \Cref{no-gap-int-of-rs-qk}, it is bounded by
$$
C_{m,\psi} h^{-l/2} \Big[ (h+t)^{-m} + |h-t|^{-m} \Big]
 \le C_{m,\psi} h^{-l/2} |h-t|^{-m}
$$
as soon as
$m\le \min(p,a+b+(l/2-1)+1) = \min( p, a+b+l/2)$.
The constant involves, by \eqref{lem-powers-of-psi-and-r}, integrals \eqref{psi-r-integrals} with $k\le m$ and $s = n-l+(l/2-1)+(l-1) = n+l/2-2$. For $l$ even, we can set $m= a+b+l/2$ and get
\begin{equation}\label{lemma-R0}
|I^{(0)}| \le C_{m,\psi} h^{-l/2} |h-t|^{-a-b-l/2}.
\end{equation}

For $l$ odd, we apply instead \Cref{int-with-r1/2} with $m=a+b+(l-1)/2$ to $I^{(0)}$, and obtain the same bound. Note that it includes implicitly $C_{m',\sigma}$ with $m'=a+b+\lfloor l/2\rfloor + l-1$, coming from the $m$-th derivatives of $q^{(l-1)}$ and $\tilde q^{(l-1)}$.

\textbf{Step 3.} To estimate $R^{(1)}(h,r)$, we write it first as
$$
         R^{(1)}(h,r) = (-1)^{l} \big[ \e^{ihr} \overline{I_1(h,r)} + \e^{-ihr} I_2(h,r) \big]
$$
with
\begin{align*}
        I_1(h,r) &= \int_0^\infty q^{(l)}(u) k_{l}(u) du,
        \\ I_2(h,r) &= \int_0^\infty  \tilde q^{(l)}(u) k_{l}(u) du.
\end{align*}
Let us treat $I_2(h,r)$ first. Both $\tilde q^{(l)}$ and $k_l$ being smooth on $[0,+\infty)$ and $\tilde q$ compactly supported, we have for any $N\ge0$, using \eqref{k_n(0)}:
\begin{align*}
        I_2(h,r) &= \sum_{j=1}^N (-1)^j \tilde q^{(l+j-1)}(0) k_{l+j}(0) + (-1)^N \int_0^\infty \tilde q^{(l+N)}(u) k_{l+N}(u) du
        \\ &= \sum_{j=1}^N c_{l,j} \tilde q^{(l+j-1)}(0) (hr)^{-(l+j)/2} + (-1)^N \int_0^\infty \tilde q^{(l+N)}(u) k_{l+N}(u) du.
\end{align*}
In the integral
\begin{align*}
        I& =\int_0^\infty e^{itr} \psi(r)\, r^{l-1} \e^{-ihr} I_2(h,r) dr
        \\&= \sum_{j=1}^N c_{l,j} h^{-(l+j)/2} \int_0^\infty \e^{i(t-h)r} \psi(r)\, r^{l/2-1-j/2} \tilde q^{(l+j-1)}(0) dr
        \\& \quad + (-1)^N \int_0^\infty \e^{i(t-h)r} \psi(r)\, r^{l-1} \int_0^\infty \tilde q^{(l+N)}(u) k_{l+N}(u) du dr,
\end{align*}
we apply \Cref{no-gap-int-of-rs-qk} or \Cref{no-gap-int-of-rs-qk-1/2} (depending on the parity of $l-j$) to every summand in~$j$, and \eqref{kn(u)}, \eqref{q-k-th-derivative-final} to the last term. This gives
\begin{align*}
        |I|& \le \sum_{j=1}^N |c_{l,j}| \, h^{-(l+j)/2} C_{m,j,\psi} |h-t|^{-m_j}
        \\&\qquad+ C_N h^{-(l+N)/2} \int_0^\infty |\psi(r)|\, r^{l-1-(l+N)/2} (1+r)^{n-l+l+N} dr
        \\& \le \sum_{j=1}^N C_{m,j,\psi} \, h^{-l/2-j/2} |h-t|^{-m_j}  + C_{N,\psi} \, h^{-l/2-N/2},
\end{align*}
with
\begin{equation}\label{mj}
m_j=a+b+(l/2-1-j/2) +1 = a+b+ (l-j)/2
\end{equation}
(it can be integer or half-integer), so that in particular $l/2+j/2+m_j = a+b+l$ for every $j$.

We choose now $N = 2(a+b) + l$ and obtain
$$
|I| 
 \le C_\psi h^{-(l+1)/2} \min( h, |h-t|) ^{-a -b- (l-1)/2}.
$$

The constants involve \eqref{psi-r-integrals} with $k\le a+b+ \lceil (l-j)/2 \rceil$ and $s=n-l+\lceil (l-j)/2-1 \rceil + (l+j-1) = n + \lceil (l+j)/2 \rceil -2$, so that finally we have $k\le a+b+ \lceil (l-1)/2 \rceil$ and $s=n + \lceil (l+N)/2 \rceil -2 = a+b+l +n-2$. In the last term, we need $k=0$ and $s=l/2+N/2+n-1 = a+b+l+n-1$.
Moreover, the constants are multiples of $C_{m',\sigma}$ with $m'=l+N = 2(a+b+l)$.

We get the same bounds for the integral with $I_1(h,r)$ (with $h+t> |h-t|$), and as a consequence,
\begin{equation}\label{lemma-R1}
\Big| \int_0^\infty e^{itr} \psi(r)\, r^{l-1} R^{(1)}(h,r) dr \Big|
\le C_\psi h^{-(l+1)/2} \min( h, |h-t|) ^{-a -b- (l-1)/2}.
\end{equation}

\textbf{Step 4.} We consider now $R^{(2)}(h,r)$ where we can choose any $M$. 
We have
\begin{align}
\begin{split}
I^{(2)} &= \int_0^\infty e^{itr} \psi(r)\, r^{l-1} R^{(2)}(h,r)\,dr \\
 &= - \frac{i^M}{h^M} \int_1^2 \int_0^\infty
    \psi(r)\, r^{l-1-M}\Big[(-1)^M e^{ ir( t+h-hv) } \overline{q_1^{(M)}(v)} \\
 &\hphantom{= - \frac{i^M}{h^M} \int_1^2 \int_0^\infty \psi(r)\, r^{l-1-M} \Big[}
    + e^{ ir(t-h+hv)} \, \tilde q_1^{(M)}(v)
    \Big] \, dr \, dv .
\end{split}
\end{align}
By \eqref{q1-k-th-derivative-small-r}, both $\overline{q_1}^{(M)}(v)$ and $\tilde q_1^{(M)}(v)$ are of order $r^{a}$ for small $r$, and by assumption, $\psi(r)$ is of order at most $r^b$, thus we can take $M = l+a+b-1$
and still get a function integrable in $r$ (around zero). For $r\ge1$, by \eqref{q1-k-th-derivative-final},
$$
r^{l-1-M} \Big( |\overline{q_1^{(M)}(v)}| + |\tilde q_1^{(M)}(v)| \Big) \le
 (1+r)^{l-1-M + (n-l+M)} = (1+r)^{n-1},
$$
so that the integral is convergent if we have $s=n-1$ and $k=0$ in \eqref{lem-powers-of-psi-and-r}.

For $v\in [1,2]$, denote
$$
f_0(v) = \int_0^\infty \psi(r)\, r^{-a-b}
        \Big[ (-1)^M e^{ ir( t+h-hv) } \overline{q_1^{(M)}(v)} + e^{ ir(t-h+hv)} \, \tilde q_1^{(M)}(v) \Big] dr
$$
(replacing $l-1-M$ by $-a-b$, with $M$ chosen as above). We pass to the variable $u=v-1$ to obtain
\begin{align*}
\int_1^2 f_0(v) dv &= \int_0^1 \int_0^\infty \psi(r)\, r^{-a-b} \Big[ (-1)^M e^{ ir( t-hu) } \overline{q_1^{(M)}(1+u)} 
        \\&\hphantom{= \int_0^1 \int_0^\infty \psi(r)\, r^{-a-b}\Big[}
        + e^{ ir(t+hu)} \, \tilde q_1^{(M)}(1+u) \Big] dr du
        \\ &= \int_0^1 \int_0^\infty e^{ ir(t+hu)} \, \psi(r)\, r^{-a-b} \tilde q_1^{(M)}(1+u) dr du
        \\&\quad + \int_{-1}^0 \int_0^\infty e^{ ir( t+hu) } \psi(r)\, r^{-a-b}
        (-1)^M \overline{q_1^{(M)}(1-u)} dr du.
\end{align*}
Denote
$$
f_1(r,u) = \begin{cases} r^{-a}\, \tilde q_1^{(M)}(1+u), & u\in [0,1];
\\ r^{-a} \,(-1)^M \overline{q_1^{(M)}(1-u)}, & u\in [-1,0).
\end{cases}
$$

Recall that $q_1$, $\tilde q_1$ are defined and smooth (in $v$) on $[1/2,3/2]$, and on this segment $\tilde q_1^{(M)} (v) = (-1)^M \overline{q_1^{(M)} (2-v)}$, by \eqref{q1-tilde-q1}. This means that for $u\in [-1/2,1/2]$, the two formulas of $f_1$ are equal, and as a whole, with \eqref{q1-k-th-derivative-small-r}, $f_1$ is smooth on $[-1,1]$.

We can also write $\psi(r) r^{-b} = \psi_1(r)$ so that $\psi_1$ is continuously differentiable $p-b\ge2$ times on $[0,+\infty)$: by 
Taylor's formula, 
$$
\psi(r) 
 = r^b \int_0^1 \psi^{(b)}(rz) \frac{ (1-z)^{b-1} } {(b-1)!} dz \equiv r^b \psi_1(r).
$$

Denote finally $f_2(r,u) = \psi_1(r) f_1(r,u)$, thus we are estimating the integral
$$
    I^{(2)}_0 = - \frac{h^M}{i^M} \, I^{(2)}  = \int_{-1}^1 \int_0^\infty e^{ ir(t+hu)} \, f_2(r,u) dr du.
$$
By \eqref{qkth-derivative-in-r}, $f_1$, $\dfrac {\partial }{\partial r}  f_1$, $\dfrac {\partial^2 }{\partial r^2}  f_1$ are bounded for $r\ge1$ by $C\, r^{-a + n-l+M}
  = C\, r^{n+b-1}$, so that 
\begin{align*}
    | f_2(r,u) | &\le C\, r^{n+b-1} | \psi_1(r) | = C\, r^{n-1} | \psi(r) |,
    \\ \Big| \frac {\partial }{\partial r}  f_2(r,u) \Big| 
    &\le C\, r^{n+b-1} ( |\psi_1'(r) | + | \psi_1(r) | )
    \le C r^{n-1} \big( |\psi'(r)|+ |\psi(r)|\big),
    \\\Big| \frac {\partial^2 }{\partial r^2}  f_2(r,u) \Big| &\le
    C r^{n-1} \big( |\psi(r)|+ |\psi'(r)|+|\psi''(r)|\big).
\end{align*}
All three $f_2$, $\dfrac {\partial }{\partial r}  f_2$, and $\dfrac {\partial^2 }{\partial r^2}  f_2$ are (continuous and) integrable on $[0,+\infty)$ if $\psi$ satisfies \eqref{psi-r-integrals} with $s=n-1$ and $0\le k\le 2$.

Set now $J = [-\frac th - \frac 1h, -\frac th + \frac 1h]$.
We split the integral 
as
$I^{(2)}_0 = I^{(2)}_1 + I^{(2)}_2$ where
\begin{align*}
I^{(2)}_1 &=\int_{J\cap [-1,1]} \int_0^\infty e^{ ir(t+hu)} \, f_2(r,u) dr du,
\end{align*}
so that $|I^{(2)}_1| \le C_\psi / h$. In $I^{(2)}_2$, we integrate by parts:
\begin{align*}
I^{(2)}_2 
&= \int_{v\in [-1+\frac th,1+\frac th] \atop |v|>1/h} \int_0^\infty e^{ irhv} \,  f_2(r,v-t/h)  dr dv
\\&= \int_{v\in [-1+\frac th,1+\frac th] \atop |v|>1/h} \frac {i}{vh}
 \Big[ f_2(0,v-t/h) + \int_0^\infty e^{ irhv} \, \frac {\partial }{\partial r}  f_2(r,v-t/h)  dr \Big] dv.
\end{align*}
We have either $0< h\le t-1$, so that $-1+t/h\ge 1/h$, or $h\ge t+1$, in which case $-1+t/h\le -1/h$. The integral is bounded in both cases as,
$$
|I^{(2)}_2| \le C_\psi \frac1h \int_{1/h}^{1+t/h} \frac {dv} v = C_\psi \frac1h \big( \log(1+t/h) + \log h \big) \le C_\psi \frac{\log t + \log h}h.
$$
Here, the last inequality follows from the estimate
\begin{align*}
	\log(t+h) \le \log(2 \max(t,h)) &\le \log 2 + \log \max(t,h)
	\\&\le \log 2 + \log h + \log t \le 2(\log h + \log t),
\end{align*}
which holds whenever $t \ge 2$ or $h \ge 2$. This condition is ensured by the assumptions $t \ge 1$, $h \ge 1$, and $|h - t| \ge 1$.
Then the estimate of $|I^{(2)}_2|$ yielding
\begin{equation}\label{lemma-R2}
|I^{(2)}| \le C_\psi h^{-M-1} (\log t + \log h) = C_\psi h^{-a-b-l} (\log t + \log h).
\end{equation}

For $h\ge 2t$ the logarithmic factor can be removed: we integrate by parts once more and obtain
\begin{align*} 
I^{(2)}_2 = \int_{v\in [-1+\frac th,1+\frac th] \atop |v|>1/h} 
 \Big[& \frac {i}{vh} f_2(0,v-t/h) - \frac {1}{v^2h^2} \frac {\partial }{\partial r}  f_2(r,v-t/h) \Big |_{r=0} 
 \\& - \frac {1}{v^2h^2} \int_0^\infty e^{ irhv} \, \frac {\partial^2 }{\partial r^2}  f_2(r,v-t/h)  dr \Big] dv.
\end{align*}
In the first term,
the function $f_3(v) = f_2(0,v-t/h)$ is continuously differentiable in $v$, and we can write $f_3(v) = f_3(0) + v f_3'(s_v)$ with $s_{v}\in [0,v]$. This implies that
\begin{align*}
\Big| \int_{v\in [-1+\frac th,1+\frac th] \atop |v|>1/h} 
 \frac {i}{vh} f_3(v) dv \Big|
  &\le \Big| f_3(0) \int_{v\in [-1+\frac th,1+\frac th] \atop |v|>1/h} \frac {1}{vh} dv \Big|
 + \frac2h \sup_{s_v\in [0,v]} |f_3'(s_v)|
 \\ &\le \frac {C |\psi_1(0)|} h \Big[ \log\frac{|1 + \frac th|}{|-1+\frac th|} + 1 \Big]
 \le C|\psi^{(b)}(0)| \frac 1h,
 \end{align*}
since $|t/h|\le 1/2$.

Next, note that
$$
\int_{v\in [-1+\frac th,1+\frac th] \atop |v|>1/h} \frac {1}{v^2h^2} dv \le 2 \int_{1/h}^\infty \frac {1}{v^2h^2} dv =
  \frac 2h,
$$
so that the two remaining terms in $I^{(2)}_2$ are bounded by $C_\psi / h$. Altogether, we get for $h\ge 2t$
$$
|I^{(2)}| \le C_\psi h^{-M-1} = C_\psi h^{-a-b-l}.
$$

Collecting \eqref{lemma-main}, \eqref{lemma-R0}, \eqref{lemma-R1}, and \eqref{lemma-R2},
we conclude that the integral \eqref{I_t-lem1-no-gap} is bounded by
$$
C_{\psi}\, h^{(1-l)/2} \min( h, |h-t|)^{-a-b-(l+1)/2} \log t  
$$
for $h<2t$ in the considered range, and by
$$
C_{\psi} h^{-(a+b+l)}
$$
if $h\ge 2t$.
As mentioned before, the constant is a multiplier of $C_{m,\sigma}$ with the largest $m$ involved in the proof, that is, $m=2(a+b+l)$ (in the estimate of $R^{(1)}$).
This proves the lemma.
\end{proof}

\begin{lemma}\label{derivative wrt r}
    Let $M>0$ be a constant and $\kappa$ a smooth function on $\R^l$, satisfying
    \begin{equation}\label{condition on kappa}
    \abso{\partial^k \kappa(\l)}\le M \bracket{1+\abso{\l}}^{p}, \qquad \l\in\R^l,
    \end{equation}
    for any multi-index $k\in \Z_+^l$, and there exists an integer $a\ge0$ such that 
    $$
    \partial^m\kappa(0) = 0
    $$
    for any multi-index $m$ with $|m|< a$.
    Then for any real $r>0$, the function 
    \[f(r):=\int_{\abso{\l}=r} \kappa(\l) d\l \]
    has the following estimate with every $m\ge0$: 
    \[
    | f^{(m)}(r)|
    \le C_{l,m} M \bracket{1+r}^{l-1+p}. 
    \]
    Moreover, $f^{(m)}(0)=0$ if $m<a+l-1$.
\end{lemma}
\begin{proof}
    In polar coordinates \eqref{polar-nots}, we rewrite $f(r)$ as
    \[
    f(r) = 
    \int_{Sk ^{l-1}}  \kappa( r\Theta ) r^{l-1} d\Theta,
    \]
    where $S^{l-1}$ denotes the unit sphere in $\R^l$. We can differentiate under the integral sign.
    First, for every $m$
    \begin{equation}\label{diff-kappa-m}
    \frac{d^{m}}{d r^{m}} \kappa( r\Theta ) =
    \sum_{k\in \Z_+^l,\ |k|=m}  \partial ^{k}\kappa(r\Theta) P_k(\theta_1,\dots,\theta_{l-1}),
    \end{equation}
    where $P_k$ are trigonometric polynomials not involving $r$. It follows that
    $$
    \Big| \dfrac{d^{m}}{d r^{m}} \kappa( r\Theta ) \Big| \le C_m M (1+r)^p
    $$
    for $\Theta\in S^{l-1}$.
    Differentiating $f(r)$ yields
    \begin{align}\label{diff-f-m-lemma}
        f^{(m)}(r) 
        & = \int_{S^{l-1}} \frac{d^m}{d r^m} \Bracket{\kappa( r\Theta ) r^{l-1}} d\Theta
        \notag
        \\& = \int_{S^{l-1}} \sum_{0\le k\le \min(m,l-1)} \binom m{k} \frac{(l-1)!}{(l-1-k)!}\,
        r^{l-1-k} \frac{d^{m-k}}{d r^{m-k}} \kappa( r\Theta )  d\Theta,
    \end{align}
    which now implies
    \[
    | f^{(m)}(r)|  
    \le C_{l,m}\, M (1+r)^{l-1+p}.
    \]
When $r=0$, only the term with $k=l-1$ may be nonzero in \eqref{diff-f-m-lemma}. If $m-k=m-l+1<a$, it vanishes as well, as it may be seen from \eqref{diff-kappa-m}. Thus, $f^{(m)}(0)=0$ if $m<a+l-1$.
\end{proof}

\begin{corollary}\label{corollary1 of derivative wrt r}
    Assume that $\xi_H(r):= \int_{\abso{\l}=r} \varphi_\l(\e^{H})\abso{\HCc(\l)}^{-2} d\l$ for any $H\in \Ca$ and $\l\in \mfa$, then for any $m$
    \[
    \abso{\xi_H^{(m)}(r)} \le C_{l,m}\, (1+\|H\|)^{m}\, \phi_0(\e^H)  (1+r)^{n-1}.
    \]
    Moreover, $\xi_H$ has a zero of order $\nu-1$ at $r=0$.
\end{corollary}
\begin{proof}
Set $\kappa(\l) = \varphi_\l(\e^{H})\abso{\HCc(\l)}^{-2}$. From \Cref{estimate of phi_lambda} and \Cref{cF&c}, we get \eqref{condition on kappa} with $p=n-l$ and $M = C_{m,l} \, (1+\|H\|)^{m}\, \phi_0(\e^H)$.
Moreover, $\abso{\HCc(\l)}^{-2}$ has a zero of order $\nu-l = 2d$ at $\l=0$ by \eqref{estiofHCc}, so the same holds for $\kappa$. \Cref{derivative wrt r} gives the result now.
\end{proof}
\begin{remark}
    A similar result for the density $\xi_H$ itself is given in \cite[\S 2, p.117]{CGM1}. We need however to estimate also all of its derivatives.
\end{remark}

\subsection{Proof of \Cref{eit-psi-LB-no-gap}}
\textbf{\\The general estimate.}
For $H\in \Ca$,
\[
 k_t(\e^H)=C \int_0^\infty e^{itr} \psi(r) \int_{|\l|=r} \varphi_\lambda(e^H) |\HCc (\lambda)|^{-2} d\lambda dr
=C\int_0^\infty e^{itr} \psi(r) \xi_H(r) dr
\]
with $\xi_H$ the same as in \Cref{corollary1 of derivative wrt r}. By this corollary, $\xi_H$ has a zero of order $\nu-1$ at $r=0$. Similarly to \eqref{I-integration by parts}, we can integrate by parts $m$ times as soon as $m\le \min(p,\nu)$, and obtain
\begin{align*}
| k_t(\e^H)| &\le |t|^{-m} 
\Big[ |( \psi \xi_H )^{(m-1)}(0)| + \int_0^\infty | ( \psi\xi_H )^{(m)}(r)| dr \Big]
\\& \le C_{l,m,\psi} \, |t|^{-m} (1+\|H\|)^{m} \, \phi_0(\e^H),
\end{align*}
the constant majorizing $|\psi^{(k)}(0)|$ with $0\le k\le m-1$ and integrals \eqref{psi-r-integrals} with $s=n-1$ and $0\le k\le m$. This proves (1) of the theorem and Remark \ref{case1,m<=nu}.

\noindent\textbf{Estimates for $\|H\| \le 3 |t|$.}
It is known \cite[Lemma 2.1.6(ii)]{AnkerJi} that there exists a constant
 $\zeta>0$ such that for every $H\in\mfa$,
\begin{equation}\label{zeta}
\max_{\a\in \Sigma_s^+} |\a(H)| > \zeta \|H\|.
\end{equation}
Set
$$
F = \{ \a\in \Sigma_s^+: \a(H) \le N \log |t| \}
$$
with $N=l/2$ (it needs not be an integer). We suppose that $\log |t|\ge1 \ge 1/N$, so that in particular $1+\log |t|\le 2\log |t|$.
If $F=\Sigma_s^+$, we have $\|H\| \le \zeta^{-1} N \log |t|$, and by the general estimate above,
\begin{align*}
|k_t(e^H)| &\le C_{l,m,\psi}\, |t|^{-m} (\zeta^{-1} N \log |t|)^m \varphi_0(\e^H) 
\\&= C_{l,m,\psi}\, |t|^{-m} (\log |t|)^m (1+\|H\|)^d \e^{-\rho(H)} 
\end{align*}
for $m\le \min(p,\nu)$. To apply it, we need \eqref{psi-r-integrals} with $s=n-1$. To put this into accordance with the estimate in (2), we choose $m=p=\lfloor l/2 \rfloor <\nu$, then
$$
|t|^{-m} (\log |t|)^{m+d}\e^{-\rho(H)} \le |t|^{-l/2} (\log |t|)^{l/2+1+ d} \e^{-\rho(H)} \le C_l\, |t|^{(1-l)/2} \e^{-\rho(H)}.
$$

Otherwise, $F\ne\Sigma_s^+$. We can decompose $H = H_0 + H_1$ so that $\a(H) = \a(H_0)$ for $\a\notin F$, and $\a(H) = \a(H_1)$ for $\a\in F$.
In detail, we could choose a basis $(H_i)$ of $\mfa$ such that $(H_i)$ is dual of $\Sigma_s^+$. Then we can define $H_1=\sum_{\alpha\in F}c_\alpha H_\alpha$ and $H_0=H-H_1$. Note also that $\|H_1\| \le \zeta^{-1} N \log |t|$.

Consider the parabolic subgroup associated to $F$ and its reductive component $M_{1F}$. We have in particular that $H_0$ is in the center of $M_{1F}$ (equivalently speaking, $M_{1F}$ is the centralizer of $H$), so that $\rho_F(H)=\rho_F(H_1)$.
Also,
$$
\tau_F(H) = \min_{\alpha\in \Sigma^+\setminus \Sigma_F^+}|\alpha(H)| > N \log |t| \ge 1.
$$
We can apply now \Cref{thm593a}:
There is a constant $C_F>0$ such that for all $\lambda\in \mfa^*$ and $a\in \CA$ with $\tau_F(\log a)\ge 1$
\begin{align*}
    \Big| \e^{\rho(H)} \varphi_\lambda(e^H) - |\!\mfw_{F}\!|^{-1} \, \e^{\rho_F(H)} \sum_{s\in \mfw}\frac{\HCc(s\lambda)}{\HCc_F(s\lambda)} \theta_{s\lambda}(e^H) \Big|
    &\le C_F \bracket{1+\abso{\lambda}}^{6(n-l)+1} \e^{-\tau_F(H)}
    \\&\le C_F \bracket{1+\abso{\lambda}}^{6(n-l)+1} |t|^{-N}.
\end{align*}
Note at this point that $C_F$ depends on $F$ only, and there is a finite number of subsets of $\Sigma_s^+$; this means that there is a unique constant majorizing all of $C_F$, thus independent of $H$.

Let us denote now
\begin{align*}
I_t(H) &= |\!\mfw_F\!|^{-1} \, \e^{\rho_F(H)} \int_{ \mfa^*} e^{it|\l|} \psi(|\lambda|) \sum_{s\in \mfw} \frac{\HCc(s\lambda)}{\HCc_F(s\lambda)} \theta_{s\lambda}(e^H) |\HCc (\lambda)|^{-2}d\lambda;
\end{align*}
by the estimate above,
\begin{align}\label{k_t-I_t}
| k_t(e^H) - C_F \, \e^{-\rho(H)}  I_t(H)| 
&\le C_F \, e^{-\rho(H)}\, |t|^{-N} \int_{ \mfa^*} \big| \psi(|\lambda|) \big| (1+|\lambda|)^{7(n-l)+1} d\lambda 
\\& \le C_{F,\psi} \, e^{-\rho(H)} |t|^{-N}. \notag
\end{align}
To guarantee the convergence of the integral, we need \eqref{psi-r-integrals} with $k=0$ and $s=7(n-l)+1$.

Our next task is to estimate $I_t(H)$. Using the fact that every $s\in \mfw$ is orthogonal and $\abso{\HCc(\cdot)}$ is $\mfw$-invariant, we obtain
 \begin{align}\label{I_t-H_0+H_1}
I_t(H) &= C_F\, \e^{\rho_F(H)} \int_{ \mfa^*} e^{it|\l|} \psi(|\lambda|) [\HCc_F\bar\HCc]^{-1}\!(\l) \, \theta_{\lambda}(e^H) d\lambda.
\end{align}
The definition of $\theta_\l$, as introduced in \Cref{definition of theta}, involves $H(e^H k)$ with $k\in K_F$. 
If $e^{H_1} k = k_1 a_1 n_1$ is the Iwasawa decomposition in $M_{1F}$, then, since $H_0$ is in the center of $M_{1F}$, we have $e^H k = e^{H_0} e^{H_1} k = k_1 e^{H_0} a_1 n_1$ and thus $H(e^H k) = H_0 + H(e^{H_1}k)$.
This implies that $\rho_F(H(e^H k)) = \rho_F(H(e^{H_1} k))$, and, simplifying notations as $H(e^{H_1}k) = (H_1)_k$, we can write $I_t$ in the following form:
\begin{align*}
    I_t(H) & = C_F\, \e^{\rho_F(H)} \int_{K_F}\int_{ \mfa^*} e^{it|\l|} \psi(|\lambda|)[\HCc_F\bar\HCc]^{-1}\!(\l) \, \e^{(i\lambda-\rho_F)(H(e^H k))} d\lambda dk
    \\ & = C_F\, \int_{K_F} \e^{\rho_F(H_1-(H_1)_k)} \int_{ \mfa^*} e^{it|\l|} \psi(|\lambda|) [\HCc_F\bar\HCc]^{-1}\!(\l) \, e^{i\lambda((H_1)_k + H_0)} d\lambda dk.
\end{align*}
For every $k$, the integral in $\l$ has the form \eqref{I_t-lem1-no-gap} with
$$
\sigma(\l) = [\HCc_F\bar\HCc]^{-1}\!(\l),
$$
its derivatives bounded by $C_m (1+|\l|)^{n-l}$ by \Cref{cF&c}. 
Moreover,
$$
h_k:=\|(H_1)_k + H_0\| \ge \max_{ \a\in \Sigma_s^+\setminus F} \frac{ \a(H_0+ (H_1)_k) }{\|\a\|} \ge C_F N \log |t|,
$$
and this is greater than 1 for $|t|$ large enough. At the same time, applying the property of function $H$ \cite[page~218, Lemma 1.14]{Helgason_Geometric_analysis}, $h_k =\norm{H(\e^H k)}\le \norm {H}\le 3|t|$.

By \Cref{stphase-oscillating-no-gap}, the integral in $\lambda$ is bounded by
$C_{\psi, m} |t|^{-m} h_k^{m+(1-l)/2}$; in the range $h_k\le 3|t|$, the choice of $m$ changes only constants and not the final degree in $t$, thus we can take simply $m=0$. Condition \eqref{psi-r-integrals} appears with $s =  n-1 + l/2$, since we can take $p=0$ in the statement of \Cref{stphase-oscillating-no-gap}.

Integration over $K_F$ turns the bound into
$$
|I_t(H)|
\le C_\psi \, |t|^{(1-l)/2} \int_{K_F} \e^{\rho_F(H_1-(H_1)_k)} dk
 = C_\psi \, |t|^{(1-l)/2} \e^{\rho_F(H_1)} \theta_0(e^{H_1}).
$$
By definition, $\theta_0$ is the spherical function at $\l=0$ of the reductive group $M_{1F}$, and it is estimated similarly to $\phi_0$, see \Cref{estimate of phi_lambda}: $\theta_0(e^{H_1}) \le \e^{-\rho_F(H_1)}(1+\|H_1\|)^{d_F}$, which yields
$$
|I_t(H)| \le C_\psi\, |t|^{(1-l)/2} (1+\|H_1\|)^{d_F}
\le C_{G,\psi}\, |t|^{(1-l)/2} (\log |t|)^d,
$$
and
\begin{align*}
| k_t(e^H) | &\le C_{G,\psi}\, e^{-\rho(H)} \big( |t|^{(1-l)/2} (\log |t|)^d + |t|^{-l/2} \big)
\\& \le C'_{G,\psi}\, e^{-\rho(H)} |t|^{(1-l)/2} (\log |t|)^d.
\end{align*}
This proves (2).

\noindent\textbf{Estimates for $\|H\|$ large and away from $|t|$.}
Let $\zeta>0$ be the same as in \eqref{zeta}; we can always assume $0<\zeta<1$.
Set this time $N = l+d+b+1$.
We consider $H$ with $h=\|H\|>\log |t|$ and $\big|h-|t| \big| \ge \delta \log |t|$ with $\delta = 2N / \zeta$.

We suppose $|t|>3 \log 3$ large enough for two inequalities: first, $\delta \log |t|\ge1$, and second, if $h>\log |t|$, then
\begin{equation}\label{h-is-good-for-lemma}
\frac 23 h > 1+ \delta \log h.
\end{equation}
Set
$$
F = \{ \a\in \Sigma_s^+ : \a(H) \le N \log h \}.
$$
We suppose $|t|$ large enough to assure $F\ne \Sigma_s^+$, and in particular, we have $|t|\ge \e$.

We decompose $H = H_0 + H_1$ in the same way as in the case $\|H\|\le 3|t|$, obtaining
\begin{equation}\label{h1-in-th4.1}
h_1 := \|H_1\| \le \zeta^{-1} N \log h = \frac\delta2 \log h
\end{equation}
and
$$
\tau_F(H)  = \min_{\alpha\in \Sigma^+\setminus \Sigma_F^+}|\alpha(H)| > N \log h \ge \log \log t > 1
$$
for $|t|$ is large enough.
With the same $I_t(H)$ \eqref{I_t-H_0+H_1} we have: first,
\begin{align}\label{kt-It<h^{-N}}
| k_t(e^H) -  C_G \, \e^{-\rho(H)}  I_t(H)| 
 \le C_{F,\psi} \, e^{-\rho(H)} \e^{-\tau_F(H)}
 \le C_{F,\psi} \, e^{-\rho(H)} h^{-N}.
\end{align}
Next, written as
$$
I_t(H)  = \e^{\rho_F(H_1)} \int_{ \mfa^*} e^{it|\l|} \psi(|\lambda|) [\HCc_F\bar\HCc]^{-1}\!(\l) \, \theta_{\lambda}(e^{H_1}) e^{i\l(H_0)} d\lambda,
$$
$I_t(H)$ has form \eqref{I_t-lem1-no-gap} with $\sigma(\l) = \e^{\rho_F(H_1)} [\HCc_F\bar\HCc]^{-1}\!(\l) \, \theta_{\lambda}(e^{H_1})$.
Identically to \Cref{estimate of phi_lambda}, every derivative of $\theta_\l$ is bounded by
\[
    \Big|\frac{\partial^m}{\partial \l^m} \theta_\l(e^{H_1}) \Big|
    \le C_m (1+h_1)^{|m|+d_F} \, \e^{-\rho_F(H_1)},
\]
and it follows that
\[
    \Big|\partial^m \sigma(\l) \Big|
    \le C_m (1+|\l|)^{n-l} (1+h_1)^{|m|+d_F}.
\]
Note that $d_F<d$ since $F\ne \Sigma_s^+$. By \eqref{h-is-good-for-lemma}, we have
\begin{equation}\label{h0-below}
h_0 := \| H_0\| \ge \|H\| - \|H_1\| \ge h - \frac\delta2 \log h >\frac32+ \delta \log h >1.
\end{equation}
To simplify the notations, denote $\tau=|t|$ till the end of the proof. If $h\ge 3\tau$, then $|h-\tau|\ge \frac23h$, and with \eqref{h-is-good-for-lemma}
$$
h_1 \le \frac \delta2 \log h < \frac13 h \le \frac12 (h-\tau). 
$$
%
Otherwise $h_1 \le \frac \delta 2 \log(3\tau) \le \frac23 \delta \log \tau \le \frac 23 |h-\tau|$, and therefore in both cases
\begin{equation}\label{h0-t}
h_1 \le \frac 23 |h-\tau|, \qquad  |h_0-\tau| \ge |h-\tau| - h_1 \ge \frac 13 |h-\tau|\ge \frac 13 \delta \log \tau.
\end{equation}
When $\tau$ is as large as assumed, the latter is in particular greater than 1.

\Cref{big-h-no-gap} now applies: The function $\sigma$ has a zero of order $a=d+d_F$ at $\l=0$, and $p=2d+\lfloor l/2\rfloor \ge d+d_F+1+\lfloor l/2 \rfloor$.
Thus, by \eqref{eq-big-h-no-gap}, we have for $h\le 3\tau$ (so that in particular $\log h_0 \le C \log \tau$)
\begin{align*}
|I_t(H)| &\le C_{\psi} C_{2(a+b+l),\sigma} \,h_0^{(1-l)/2} \min(h_0, |h_0-\tau|)^{-a-b-(l+1)/2} \log \tau
\\ & \le C_{\psi} \,(1+h_1)^{2(d+d_F+b+l)+d_F}\, h_0^{(1-l)/2} \min(h_0,|h_0-\tau|)^{-d-d_F-b-(l+1)/2} \log \tau.
\end{align*}
We estimate with \eqref{h-is-good-for-lemma} that $h_0\ge h-\frac{\delta}{2}\log h > h- \frac{1}{3}h =\frac{2}{3}h$. Thus with \eqref{h0-t} we have first
\[
\min(h_0, |h_0-\tau|)\succsim \min (h,|h-\tau|).
\]
It follows from $h_0 \le h+h_1\le 2h$ and \eqref{h0-below} that
\begin{align}\label{It-with-d_F}
&|I_t(H)| 
\\&\le C_{\psi} \,(1+\log h)^{2(d+d_F+b+l)} h_1^{d_F}\, h^{(1-l)/2} \min(h,|h-\tau|)^{-d-d_F-b-(l+1)/2} \log \tau \notag
\\ &\le C_{\psi} \,(1+\log h)^{2(d+d_F+b+l)}\, h^{(1-l)/2} \min(h,|h-\tau|)^{-d-b-(l+1)/2} \log \tau
\notag
\\&\le C_{\psi} \, h^{(1-l)/2} \min(h,|h-\tau|)^{-d-b-(l+1)/2} (1+\log h)^{2(2d+b+l)} \log \tau.
\notag
\end{align}
The constant involves \eqref{psi-r-integrals} with
$s = n+l/2-1 + ( d+d_F + \lfloor l/2 \rfloor+1) /2$.
However, for the estimate \eqref{k_t-I_t} we need $s = 7(n-l)+1$.

Finally, with \eqref{kt-It<h^{-N}} in which, as we recall, $N>d+b+l$, we have for $\tau=|t|$ large enough and $h\le 3\tau$:
\begin{align*}
| k_t(e^H) | 
&\le C_{G,\psi}\, e^{-\rho(H)} \big( h^{-d-b-l} 
\\& \qquad + h^{(1-l)/2} \min(h,|h-\tau|)^{-d-b-(l+1)/2}\big) (\log \tau)^{2(2d+b+l)+1}
\\&\le C_{G,\psi}\, e^{-\rho(H)} \big( h^{-d-b-l} 
\\&\qquad+ h^{(1-l)/2} (h^{-d-b-(l+1)/2}+|h-\tau|^{-d-b-(l+1)/2})\big) (\log \tau)^{2(2d+b+l)+1}
\end{align*}
which proves (4).

For $h\ge 3\tau$, we have $h_0\ge 2\tau$ and Lemma \ref{big-h-no-gap} gives a better result: 
\begin{align}\label{result-of-lemma-in-th-4.1}
|I_t(H)| &\le C_{\psi} C_{2(a+b+l),\sigma} \,h_0^{-a-b-l}
\le C_{\psi} \,(\log h)^{2(d+d_F+b+l)+d_F}\, h^{-d-d_F-b-l},
\end{align}
If $d_F\ge1$, then the logarithmic factor is smaller than $h\le h^{d_F}$ for $\tau$ large enough; but if $d_F=0$, then $H_1=0$ and this factor is absent. Thus, we have finally (3):
\begin{align*}
| k_t(e^H) | &\le C_\psi\, e^{-\rho(H)} h^{-d-b-l}.
\end{align*}
\qed

\section{$L^{p'}\to L^p$ estimates for the Laplace-Beltrami operator}

Our detailed pointwise estimates allow to deduce multiple properties, of which we present two most interesting ones: dispersive estimates for the Laplace-Beltrami operator in this section and $L^p-L^p$ estimates for the distinguished Laplacian in Section 6.

For $1\le p\le \infty$, let us denote by $p'$ its conjugate exponent: $1/p+1/p'=1$.
For right convolution operators with $K$-biinvariant functions, the following Kunze-Stein phenomenon gives better estimates than a direct $L^2-L^\infty$ interpolation:
\begin{lemma}\label{KunzeStein}
    Assume that $\kappa$ is a $K$-biinvariant measurable function on $G$. For any $p\in [2,\infty)$, there exists a constant $C_p$ such that for each $f\in L^{p'}(S)$,
    \[
    \norm{f*\kappa}_{L^p(S)} \le C_p \norm{f}_{L^{p'}(S)} \set{\int_G  \abso{\kappa(x)}^{p/2}\varphi_0(x)dx }^{2/p}.
    \]
\end{lemma}
\begin{proof}
    See \cite[Lemma 4.1]{AnkerZhang} and \cite[Theorem 4.2]{AnkerVittoriaVallarino}.
\end{proof}

We can apply now the results above:
\begin{theorem}[Dispersive estimate for $\Psi_t(\Delta_\rho)$]\label{dispersive-LB-no-gap}
Suppose that $\Psi_t(r^2) = e^{itr} \psi(r)$ with $\psi$ satisfying assumptions of Theorem \ref{eit-psi-LB-no-gap}. Then
    \[
    \norm{\Psi_t(\Delta_\rho)}_{L^{p'}(S)\to L^{p}(S)} \le C_{G,\psi}\, |t|^{-\nu}
    \]
    for $t\ne0$ and any $2< p< \infty$.
\end{theorem}
\begin{proof}       
By \Cref{estimate of phi_lambda}, \eqref{integral_formula_KAK} and \Cref{eit-psi-LB-no-gap} (1)
we have then
\begin{align*}
    \begin{split}
        &\set{\int_G  \abso{k_t(x)}^{p/2}\varphi_0(x)dx }^{2/p}
        \\&\le C \set{\int_{\mfa^+}  \abso{k_t(\exp x^+)}^{p/2}\varphi_0(\exp x^+) D(x^+) dx^+ }^{2/p}
        \\&\le C |t|^{-\nu} \set{\int_{\mfa^+}  (1+|x^+|)^{p(d+\nu)/2}  e^{ -(p/2)\rho(x^+)}\varphi_0(\exp x^+) D(x^+) dx^+ }^{2/p}
        \\&\le C |t|^{-\nu} \set{\int_{\mfa^+}  (1+|x^+|)^{p(d+\nu)/2+d}  e^{ (-p/2-1+2)\rho(x^+)} dx^+ }^{2/p}
        \\&\le C |t|^{-\nu},
    \end{split}
\end{align*}
provided $p/2-1>0$, i.e., $p>2$.
Applying \Cref{KunzeStein}, we obtain the $L^{p'}\to L^{p}$ estimate of $\Psi_t(\Delta_\rho)$ for any $2< p< \infty$.
\end{proof}

This theorem, as well as Theorem \ref{Lp-Lq-LB} below, applies, in particular, to the function $\psi(x) = (1+x^2)^{-\a/2}$ with $\Re\a>n$; the decay conditions are imposed only by (1) of Theorem \ref{eit-psi-LB-no-gap}.

This should be compared with the result \eqref{kt-AZ dispersive} of Anker and Zhang \cite{AnkerZhang} cited in the introduction. As $|\Delta| = \Delta_\rho + |\rho|^2$, their function is $\Psi(\Delta_\rho)$ with $\Psi(x) = (x+|\rho|^2) ^{-\a/2} \exp( it\sqrt{x+|\rho|^2} )$. Moreover, the result still holds for $\Psi_\varepsilon(x) = (x+\varepsilon) ^{-\a/2} \exp( it\sqrt{x+\varepsilon} )$ with any $\varepsilon>0$, as mentioned in Section~6 of \cite{AnkerZhang}. In \cite{AnkerZhang}, the emphasis is on the case $\Re \a=\dfrac{n+1}2$ when the integrals \eqref{psi-r-integrals} do not converge. However, even the stronger requirement $\Re\a>n+1$ does not give a decay faster than $t^{-\nu/2}$ in \eqref{kt-AZ dispersive}, as $t\to+\infty$; the authors conjecture that this exponent is optimal. 

Our theorems do not apply directly to this case, but with minimal modifications the proofs can be adapted to handle it (with $\a$ large enough), and yield then the same power $t^{-\nu/2}$. This is another argument to suppose optimality of the result of Anker and Zhang, and this suggests also that we observe a jump in behaviour, from $t^{-\nu/2}$ to $t^{-\nu}$, between the case $\varepsilon>0$ of spectral gap $\varepsilon>0$ and our case $\varepsilon=0$.

One can get even more, using results of Cowling, Giulini and Meda.
A particular case of their \cite[Theorem 2.2]{CGM1} reads:
\begin{theorem}\label{CGM-duke}
Suppose that
\begin{equation}\label{pqr}
2 \le r< \infty, \quad 1\le p\le r'<2<r \le q\le \infty, \quad \frac 1p-\frac1q \le \frac1{r'} \text{ and } (p,q) \ne (r',r).
\end{equation}
Then for every $k\in L^r(G)$, the convolution operator $T_k: f\mapsto k*f$ is bounded from $L^p(G)$ to $L^q(G)$, with norm $\|T_k\|\le C_{p,q,r} \|k\|_r$.
\end{theorem}
\begin{remark}
    The proof starts with the Kunze-Stein phenomenon \cite{CowlingKS} and applies then multilinear interpolation \cite{bergh1976interpolation}.
\end{remark}

\begin{theorem}\label{Lp-Lq-LB}
Suppose that $\Psi_t(r^2) = e^{itr} \psi(r)$ with $\psi$ satisfying assumptions of Theorem \ref{eit-psi-LB-no-gap}. Then $\Psi(\Delta_\rho)$ is bounded from $L^p(S)$ to $L^q(S)$ for all $p,q$ such that $1\le p < 2<q \le\infty$, and
    \[
    \norm{\Psi(\Delta_\rho)}_{L^{p}(S)\to L^{q}(S)} \le C_{G,\psi}C_{G,p,q}\, |t|^{-\nu}
    \]
    for $t\ne0$, where $C_{G,\psi}$ is the maximum of $C_{\psi,k,n-1}$ in \eqref{psi-r-integrals} over $0\le k\le \nu$.
\end{theorem}
\begin{proof}
Let $k_t$ be the kernel of $\Psi(\Delta_\rho)$ (which is $K$-bi-invariant). We use the bound (1) of Theorem \ref{eit-psi-LB-no-gap}, where as we recall the constant is $C_{G,\psi}$ as in the statement above.
By \eqref{integral_formula_KAK} and~\eqref{est-D}, for every $r>2$
\begin{align*}
  \|k_t\|_r^r &\le C_G \int_{ \mfa^+} | k_t(\exp(H)) |^r e^{ 2\rho(H) } d H
  \\ &\le C_{G, \psi} |t|^{-\nu r} \int_{ \mfa^+} (1+|H|)^{r(d+\nu)} e^{ (2-r) \rho(H) } d H < \infty.
\end{align*}
Theorem \ref{CGM-duke} implies now the statement: it suffices to choose $r$ so that \eqref{pqr} holds.
\end{proof}  

\begin{remark}
In our assumptions, $\Psi(\Delta_\rho)$ is also bounded on $L^2(S)$ with norm $\|\psi\|_\infty$, but this norm is independent of $t$. One can derive $L^p-L^2$ and $L^2-L^q$ bounds from it, for $p<2<q$.
\end{remark}

\section{Poisson kernel}\label{sec-poisson}

This is the kernel $p_\tau$ of the operator $\cal P_\tau = \exp( -\tau \sqrt{|\Delta_\rho|})$, with complex $\tau$ such that $\Re\tau\ge0$. For $t$ real (and positive), pointwise bounds have been obtained by Anker \cite{AnkerDuke}, and $L^{1,2}$ norms estimated by Cowling et al \cite{CGM2}. For complex time $\tau$, Cowling et al \cite{CGM3} have obtained pointwise and $L^p-L^q$ estimates in the case when either $\tau$ is bounded, or $\Re\tau > \Im\tau$:
\begin{theorem}[\cite{CGM3}]
For $\tau = u+it$ complex with $u = \Re\tau>0$, the Poisson kernel $\cal P_\tau$ is bounded from $L^p(S)$ to $L^q(S)$ for all $p,q$ such that $p \le 2\le q $ and $(p,q)\ne(2,2)$. The norms are bounded as
    \[
    \norm{\cal P_\tau}_{L^{p}(S)\to L^{q}(S)} \lesssim \begin{cases} u^{-n(1/p-1/q)}, & u\le 1;\\
   u^{-\nu/2}, & u\ge1 \text{ and } p=2 \text{ or }q=2;\\
   u^{-\nu}, & u\ge1 \text{ and } p<2<q.
\end{cases}
    \]
    for $t\ne0$.
    Moreover, if $T>0$ is fixed and $|\tau|\le T$, then
 \begin{equation}\label{pq-small-tau}
   \norm{\cal P_\tau}_{L^{p}(S)\to L^{q}(S)} \sim \frac{ |\tau|^{\frac{n-1}2 \, ( |\frac1q - \frac1{p'} | - |\frac1p - \frac1q| )} }
   { u^{\frac{n-1}2 \, |\frac1q - \frac1{p'} | + \frac{n+1}2 \, |\frac1p - \frac1q| } }
 \end{equation}
 for $1\le p\le 2\le q\le \infty$.
 \end{theorem}
These estimates are shown to be optimal if $p=2$ or $q=2$, or if $p<2<q$ and with a fixed $T$,
$$
\tau \in H_T := \big\{ \tau = u+it \in\C: u\ge0, \ u^2 \ge \frac { t^2}{T^2} + 1 \big\}.
$$
It is clear that for $u\ge |t| \ge1$, or more generally $\tau\in H_T$ with a fixed $T$, the estimate \eqref{pq-small-tau} is weaker than $u^{-\nu}$. The interesting remaining case is therefore $|t| \ge u$.

Our Theorem \ref{Lp-Lq-LB} implies the following bounds:

\begin{theorem}\label{Lp-Lq-Poisson}
For $\tau = u+it$ complex with $u = \Re\tau>0$, the Poisson kernel $\cal P_\tau$ is bounded from $L^p(S)$ to $L^q(S)$ for all $p,q$ such that $p < 2<q $, and
    \[
    \norm{\cal P_\tau}_{L^{p}(S)\to L^{q}(S)} \le C_{G,p,q}\, |t|^{-\nu} \begin{cases} u^{-n}, & u\le 1;\\
u^{\nu-1}, & u\ge1.
\end{cases}
    \]
    for $t\ne0$.
\end{theorem}
\begin{proof}
We apply Theorem \ref{Lp-Lq-LB} with $\psi(x) = e^{ -ux}$, having $u>0$. 
One can verify that the integrals in \eqref{psi-r-integrals} are bounded as
$$
C_{\psi,k,n-1} \le C_n \begin{cases} u^{k-n}, & u\le 1;\\
u^{k-1}, & u\ge1,
\end{cases}
$$
which implies immediately the statement.
\end{proof}


For fixed $u=\Re\tau$, this gives a very satisfactory estimate, probably optimal at large $t$, and of course better than \eqref{pq-small-tau}. 
For $|t|\le u$ this is much weaker than $u^{-\nu}$ above. But for $|t|\ge u^\gamma$ with $\gamma>1$, we get
$$
\norm{\cal P_\tau}_{L^{p}(S)\to L^{q}(S)} \lesssim |t|^{- \nu (1/\gamma-1)} \lesssim |\tau|^{- \nu (1/\gamma-1)},
$$
which demonstrates a rapid decay and might be optimal.

We note that for small $t$, our estimates are almost certainly not sharp.


\medskip
Analysis of the Poisson kernel associated to $\cal L$ is more complicated. In the following lemma we obtain pointwise bounds which will imply that the kernel is in $L^1(S)$; and in the next section, we deduce from it that the kernel of $\Psi(\cal L)$ is always in $L^1(S)$ when $\Psi$ satisfies our usual assumptions. Final bounds for the $L^1$ norm of the Poisson kernel are obtained in Remark \ref{poisson-L-l1-norm}.

\begin{lemma}\label{lemma-poisson-ptwise}
For $\Psi(r^2) = \exp(-itr) e^{-r}$ and $\|H\|> e^{|t|}$, 
the kernel $k_{t,\cal L}$ of $\Psi(\sqrt {\cal L})$ satisfies
$ |k_{t,\cal L}(e^H)| \lesssim h^{-d-l-1}\log h $ for $|t|$ large enough.
\end{lemma}
\begin{proof}
Denote $\tau=|t|$ and suppose that $h = \| H \| \ge e^\tau$. We set $N=d+l+1$ and $\delta=2N/\zeta$ where $\zeta$ is fixed as in \eqref{zeta}.

Decomposing $H=H_0+H_1$ as in the part \textbf{Estimates for $H$ large and away from $t$} of the proof of Theorem \ref{eit-psi-LB-no-gap}, we arrive at the integral
$$
I_t(H)  = \int_{ \mfa^*} e^{it|\l|} \psi(|\lambda|)e^{i\l(H_0)}\sigma(\lambda) d\lambda
$$
where $\sigma(\lambda):= \Bracket{\HCc_F\bar\HCc}^{-1}\!(\l) \, \theta_\lambda(\e^{H_1}) \e^{\rho_F(H_1)} $, with the estimate \eqref{kt-It<h^{-N}} for $k_t(e^H)$.

The reasoning in Theorem \ref{eit-psi-LB-no-gap} is written for $h\le 3\tau$, thus attention is needed to ensure that it applies in the current case. We have clearly \eqref{h-is-good-for-lemma}, and by construction, \eqref{h1-in-th4.1}. Next, \eqref{h0-below} holds, and moreover, $2\tau<\frac23 h<h_0<2h$. Lemma \ref{big-h-no-gap} is thus applied in the assumption $h_0>2\tau$ and gives \eqref{result-of-lemma-in-th-4.1}.


If $d_F\ge1$, that is, $H$ is ``close to the walls'', we obtain from it
\begin{align*}
|I_t(H)| &\lesssim (\log h)^{5d+2l}\, h^{-d-l-1}
\end{align*}
and as a consequence
\begin{align}\label{It-if-d_F>0}
|k_{t,\cal L}(e^H)| &= e^{\rho(H)} |k_t(e^H)| 
\lesssim (\log h)^{5d+2l} h^{-d-l-1}.
\end{align}

From now, suppose that $d_F=0$ so that $H_1=0$, $H=H_0$ and $\sigma(\lambda)= \bar\HCc^{-1}\!(\l)$.
By \eqref{gindikin-k}, the $\Gamma$-function having a simple pole at $0$, we can write
\begin{equation}\label{P(lambda)}
\sigma(\lambda) = P(\l)c_1(\l), \qquad P(\l) := \prod_{\a\in\Sigma^+_r} \langle\l,\a\rangle 
\end{equation}
with a smooth function $c_1$. Set now 
\[
\sigma_1(\l) := P(\l) c_1(0)
\quad \text{and}\quad \sigma_2(\l) :=\sigma(\lambda) - \sigma_1(\l)=P(\l)(c_1(\l)-c_1(0)).
\] 
Let $I_1(t,H)$ and $I_2(t,H)$ be $I_t(H)$ above with $\sigma(\lambda)$ replaced by $\sigma_1$ and $\sigma_2$ respectively. 

To $I_2(t,H)$, Lemma \ref{big-h-no-gap} applies again, since $\sigma_2$ has a zero of order $a=d
+1$ at $\l=0$, and $h > 2\tau$, $|h-\tau|\ge1$; moreover, $\sigma_2(\l)$ and its derivatives are bounded by 
$
(1+|\l|)^{n-l}
$.
By this lemma,
\begin{equation}\label{I_2-e-r}
|I_2(t,H)|\le C\, h^{-d-l-1}.
\end{equation}

Now, 
\begin{align*}
I_1(t,H)  &= c_1(0) \int_{ \mfa^*} e^{it|\l|} \psi(|\lambda|) P(\l) \, e^{i\l(H)} d\lambda
\end{align*}
can be written, changing variables similarly to \eqref{integral-with-lambda1} and then passing to polar coordinates, as
\begin{align}\label{I_1-e-r}
I_1(t,H)  &= c_1(0) \int_0^\infty e^{itr} e^{-r}  r^{d+l-1} \int_{S^l}\, e^{irh \cos\theta_1}
J(\Theta) P(O\Theta) d\Theta dr,
\end{align}
where $O$ is an orthogonal matrix and  $J(\Theta) = \bracket{\sin \theta_{l-2}}\dots  \bracket{\sin \theta_1}^{l-2}$. 

Let us denote
\begin{equation}\label{R_1-def}
R_1(\theta_1) = \int_{[0,\pi]^{l-3}\times[0,2\pi]} J(\Theta) P(O\Theta) d\theta_2\dots d\theta_{l-1}
\end{equation}
and $R(v) = R_1(\arccos v) / \sqrt{1-v^2}$, $v\in[-1,1]$. Since $P$ is a polynomial (homogeneous of degree $d$), 
$R_1$ is a trigonometric polynomial, and $R$ is a linear combination of terms $v^k (1-v^2)^{m/2}$, with 
\[
0\le k \le d,\  -1\le l-3\le m\le d+l-3.
\] 
Changing variable to $v=\cos\theta_1$, we get
\begin{align*}
I_1(t,H)  &= c_1(0) \int_0^\infty e^{itr} e^{-r}  r^{d+l-1} \int_0^\pi R_1(\theta_1) e^{ihr\cos \theta_1} d\theta_1 dr
\\ &= c_1(0) \int_{-1}^{1}R(v) \int_0^\infty e^{(it+ihv-1) r}  r^{d+l-1}  dr dv.
\end{align*}

Denote $D = d+l$. Direct integration by parts in $r$ gives
\begin{equation}\label{I_1-h-D}
I_1(t,H) = C h^{-D} \int_{-1}^1\, R(v) \Big(v + \frac{t+i}h \Big)^{-D} dv =: C h^{-D} I(h).
\end{equation}
If we denote for $\eps\ge0$
$$
I_\eps(h) = \int_{-1+\eps}^{1-\eps}\, R(v) \Big(v + \frac{t+i}h \Big)^{-D} dv,
$$
then clearly $I(h) = \lim\limits_{\eps\to0+} I_\eps(h)$, with $| I(h) - I_\eps(h)| \lesssim \sqrt\eps$, uniformly in $h$. (Here and in the subsequent proof $\lesssim$ means up to a constant independent of $\eps$ and $h$.)

For $z$ in the unit disk $\D = \{ z\in\C: |z|<1 \}$, one has $\Re(1-z^2)>0$, so that the principal square root is well defined and holomorphic. It follows that $1/\sqrt{1-z^2}$ and $R$ are holomorphic in $\D$,
and they are also continuous on its closure $\bar\D$, except for $\pm1$. 
For fixed $ \eps\in(0,1/2)$ and $h$, the function $f_h(z) = R(z) \big(z + \frac{t+i}h \big)^{-D}$ is well defined and holomorphic in the ``almost half-disk'' $\{ z: |z|<1-\eps/2, \Im z > -\frac1{2h}\}$. Its integral over $[-1+\eps,1-\eps]$ is there\-fore equal to the integral over the following half-circle: $\Gamma_\eps = \{ (1-\eps) e^{is} : \pi \ge s\ge 0\}$.
For $z\in\Gamma_\eps$, we have $\big(z + \frac{t+i}h \big)^{-D} \to z^{-D}$, $h\to+\infty$, and the convergence is uniform:
\begin{equation}
\Big| \big(z + \frac{t+i}h \big)^{-D} - z^{-D} \Big| \le C_D \frac{ |t|+1}h \le C'_D \frac {\log h}h, \qquad z\in \Gamma_\eps.
\end{equation}
Thus,
$$
\int_{\Gamma_\eps} f_h(z) dz \to \int_{\Gamma_\eps} R(z) z^{-D} dz=: I_\eps, \quad h\to+\infty,
$$
with $ \Big| \int_{\Gamma_\eps} f_h(z) dz - I_\eps \Big| \lesssim \frac {\log h}h \int_{\Gamma_\eps} |R(z) | dz \lesssim \frac {\log h}h$.

Consider now the unit upper half-circle $\Gamma = \{ e^{is} :  \pi \ge s\ge 0 \}$ and denote $ f(z) = R(z) z^{-D}$, $I = \int_{\Gamma} f(z) dz$. We have
$$
| I - I_\eps | \le \int_\Gamma \big| f\big( (1-\eps) z \big) - f(z) \big| dz.
$$
Let $\Gamma_{\pm1}$ be the part of the curve close to the ends: 
$$
\Gamma_{\pm1} = \Gamma\cap \{ z: \min(|z-1|,|z+1|)<\eps^{5/8}\}.
$$
Since for $z\in\Gamma$ we have always $|(1-\eps)z\pm1| \ge\eps$ and therefore $|f\big( (1-\eps) z \big)| \lesssim \eps^{-1/2}$,
$$
\int_{\Gamma_{\pm1}} \big| f\big( (1-\eps) z) \big) - f(z) \big| dz \lesssim \eps^{5/8-1/2} + \int_0^{\eps^{5/8}} s^{-1/2} ds \lesssim \eps^{1/8}.
$$
For the remaining part we have $|z\pm1| \ge \eps^{5/8}$ and also $|sz\pm1|\gtrsim \eps^{5/8}$ for $s\in [1-\eps,1]$, thus for $z\in \Gamma\setminus \Gamma_{\pm 1}$
\begin{align*}
\big| f\big( (1-\eps) z) \big) - f(z) \big| &\le \eps \sup_{s\in (1-\eps,1)} | f'(sz)|
\\&\lesssim \eps \sup_{s\in (1-\eps,1)} |1-s^2z^2|^{-3/2}
\\&\lesssim \eps^{ 1 - 15/16} = \eps^{1/16} .
\end{align*}
Altogether, this implies that
\begin{equation}\label{I-formula-with-Gamma}
I(h) \to I = \int_{\Gamma} R(z) z^{-D} dz, \quad h\to+\infty,
\end{equation}
and $| I(h) - I | \lesssim \frac {\log h}h$.

Let us now evaluate $I$ more explicitly. For $|z|<1$, $R(z)$ is equal to the sum of its Taylor series:
$$
R(z) = \sum_{m=0}^\infty R_m z^m.
$$
Let us denote by $T(z)$ its Taylor polynomial of order $D-1$ and by $\tilde R(z)$ the remainder. We can note that $\tilde R$ is of order $O(z^D)$ at $z\to0$, thus $\tilde R(z)/z^D$ is holomorphic in the open unit disk.

The integral of $T(z) z^{-D}$ is calculated directly:
\begin{align*}
\int_{\Gamma} T(z) z^{-D} dz & = \int_{\Gamma} \sum_{m=0}^{D-1} \frac{ R^{(m)}(0) }{ m! } z^{m-D} dz
\\&= \sum_{m=0}^{D-1} R_m i \int_{\pi}^0 e^{(1+m-D) is} ds
\\&= \sum_{m=0}^{D-2} R_m i \Big[ \frac{ e^{(1+m-D) is} }{ i(1+m-D) } \Big]_{\pi}^0 + i\pi R_{D-1}
\\&= \sum_{m=0}^{D-2} R_m \frac{ 1 - e^{(1+m-D) i\pi} }{ (1+m-D) }  + i\pi R_{D-1} .
\end{align*}
In the sum, there remain only the terms with $m-D$ even, so that $1 - e^{(1+m-D) i\pi} = 2$. Returning to the integration over $[-1,1]$ for $\tilde R$, we obtain:
\begin{equation}\label{I-formula-with-R_D}
I = 2 \sum_{0\le m\le D-2 \atop m-D \text{ even}} \frac{ R_m }{ 1+m-D } + i\pi R_{D-1} + \int_{-1}^1 \tilde R(v) v^{-D} dv.
\end{equation}

To continue, we need more information on the function~$R$. The following reasoning has sense for $l\ge3$ only; for $l=2$ we get no additional constraints in this way.

By \eqref{P(lambda)} and \eqref{polar-nots}, $P(O\Theta)$ is a linear combination of
\begin{align*}
\cos^k\theta_1 \sin^{m_1}\theta_1 \cos^{k_2}\theta_{2} \sin^{m_2}\theta_{2} \dots\cos^{k_{l-1}}\theta_{l-1} \sin^{m_{l-1}}\theta_{l-1},
\end{align*}
where $m_j = k_{j+1}+m_{j+1}$ for every $j$ between $1$ and $l-2$. Every term is multiplied by $J(\Theta)$ and integrated, as in \eqref{R_1-def}, over $\theta_j\in[0,\pi]$, $2\le j\le l-2$, and $\theta_{l-1}\in[0,2\pi]$. This integral is nonzero only if every $k_j$ is even, $2\le j\le l-1$, and moreover $m_{l-1}$ is even.
But then by summation every $m_j$ is even as well, $1\le j\le l-2$. Since $P$ is a homogeneous polynomial of degree $d$, we also have $k+m_1=d$, which implies that $k-d$ is even.

Now, the function
\begin{equation}\label{R/v^D}
\frac{ R(v) }{ v^D } = \sum_k c_k v^{k-d-l} (1-v^2)^{(d-k+l-3)/2}
\end{equation}
(with $l-2$ in the power of $1-v^2$ coming from $J(\Theta)$)  is odd if $l$ is odd, and is even if $l\ge4$ is even; recall that for $l=2$ these considerations do not apply and $R$ can have both odd and even terms.

If $l$ is odd, then $d-k+l-3$ is even, and $R$ is actually a polynomial, of degree at most $d+l-3=D-3$. In particular, in this case $R^{(D-1)}\equiv0$ and $R_{D-1}=0$. If $l\ge4$ is even, then $R^{(D)}$ is even, $R^{(D-1)}$ odd and $R_{D-1}=0$ as well.

As for $l=2$, we write $R_e = \sum\limits_{k-d \text{ even}} c_k v^{k} (1-v^2)^{(d-k-1)/2}$ (with coefficients as in \eqref{R/v^D}) and $R = R_e + R_o$. By the same reasoning as in the case of $l\ge 4$ even, we have $R_e^{(D-1)}(0)=0$. As for $R_o$, it is a polynomial of degree $\le D-2$, as for $l$ odd, thus $R_o^{(D-1)}\equiv0$.

We conclude that in \eqref{I-formula-with-R_D}, we have always $R_{D-1}=0$.
If $l$ is odd, the sum vanishes as well. Moreover, $\tilde R$ has the same parity as $R$ so that $\tilde R(v) / v^D$ is odd, which means that $I=0$. This applies also to $R_o$ in the case $l=2$, thus in \eqref{I-formula-with-R_D} we can replace $R$ by $R_e$.

We can suppose from now that $l$ is even and $R(v) / v^D$ is even as well. We will show that $I=0$. Combining the initial formula \eqref{I-formula-with-Gamma} with \eqref{R/v^D}, we see that it suffices to show that for any integer $m\ge1$
$$
J_m := \int_\Gamma z^{-2m} (1-z^2)^{(2m-3)/2} dz = 0.
$$
Changing variables to $w=z^2$, we pass to the integral over the circle $\Gamma'= \{ e^{iu}: u\in[ 2\pi, 0]\}$
 (followed in the clockwise direction):
$$
J_m = \frac12\int_{\Gamma'} w^{-m-1/2} (1-w)^{m-3/2} dw.
$$

In this integral we recognize an analytic continuation of the Beta function (this is \cite[1.6(9)]{BE}, up to changing the contour direction and writing $-w=e^{-i\pi}w$), valid for $\Re z_2>0$ and non-integer~$z_1$:
$$
B ( z_1, z_2) = \big( 1 - e^{ 2\pi i z_1 } \big)^{-1} \int_{\Gamma'} w^{z_1-1} (1-w)^{z_2-1} dw,
$$
so that $J_m = \big( 1 - e^{ -2\pi i (m+\frac12) } \big) B(\frac12-m, m-\frac12)$.

 The expression of Beta in terms of Gamma function $B ( z_1, z_2) = \dfrac{ \Gamma(z_1) \Gamma(z_2)}{ \Gamma(z_1+z_2)}$ implies that $B(z_1,z_2)=0$ if $z_1+z_2\le 0$ is integer; and this is exactly our case, thus $J_m=0$.

We conclude that $|I_1(t,H)|\lesssim h^{-d-l-1}\log h$. This dominates also the term $h^{-d-l-1}$ appearing in \eqref{kt-It<h^{-N}} and \eqref{I_2-e-r}, as well as the estimates \eqref{It-if-d_F>0} valid close to the walls. This completes the proof.
\end{proof}
\section{$L^1$ norms for the distinguished Laplacian}\label{sec-L1}

It is known that functions of the distinguished Laplacian are often bounded on $L^p$, $1<p<\infty$ \cite{Hebisch,cowling1994spectral,sikora2002spectral}, and as mentioned in the introduction, results of Sikora~\cite{sikora2002spectral} predict for our function the estimate $\lesssim |t|^{\nu/2+\varepsilon}$ of the weak $L^1$ norm of $\|\Psi_t(\cal L)\|$, for large $t$.
In his unpublished thesis \cite{gadzinski}, Gadzinki obtained a better estimate 
\[
\|k_{\Psi(\cal L)}\|_1 \lesssim |t|^{\frac{\nu-1}2}
\]
for the special function $\Psi(x^2) = \cos( t x) \exp(-x^2)$, suggesting that oscillating functions behave better than general ones.

In this section we prove a stronger estimate $\|\Psi_t(\cal L)\|_1\lesssim |t|^{(\nu-1)/2}(\log |t|)^c$, with a constant $c>0$, for $\Psi(x^2) = e^{it x} \psi(x)$ with any function $\psi$ decreasing polynomially as in \Cref{eit-psi-LB-no-gap}, and derive from it corresponding $L^p-L^p$ estimates for every $p\ge1$.

In particular (see Remark \ref{minimal_alpha}) this applies to $\psi(x)=(1+x^2)^{-\alpha/2}$ when
        \[
            \Re\alpha > \max\!\left(7(n-l)+2,\; n+2d+(l+\lfloor l/2 \rfloor -1)/2 \right).
        \]

\begin{proposition}\label{integral_formula_L1_norm}
    For any measurable set $E\subset \mfa^+$, define $KE:=K\exp E$, which is a $K$-invariant measurable set of $S$. Then we have
    \[
    \norm{\chi_{KE} k_{t,\cal L}}_{L^1(S, d_l)} =C_G\int_{E} \varphi_0(\exp(H)) \abso{k_{t}(\exp(H))} D(H)d H.
    \]
\end{proposition}
\begin{proof}
    Set $f:=k_{t,\cal L}$ and $g:=k_t$ in \cite[Lemma 1.3]{cowling1994spectral}.
\end{proof}

In practice, this allows to estimate norms very directly:
\begin{corollary}\label{integral_formula_L1_norm}
    If $E= \{ x\in S: a< |x^+|< b\}$ with $-\infty\le a<b\le+\infty$, and $|k_{t,\cal L}(x)| \le f(|x^+|)$ for every $x\in E$, then
    \[
    \norm{\chi_{E} k_{t,\cal L}}_{L^1(S, d_l)} \lesssim \int_a^b f(h) (1+h)^d h^{l-1} dh.
    \]
\end{corollary}
\begin{proof}
For $x=\exp(H)\in E$ with $H\in \mfa^+$ we have $|k_t(x)| = e^{-\rho(H)} |k_{t,\cal L}(x)| \le e^{-\rho(H)} f(|H|)$. For the functions $\phi_0$ and $D$ we have known estimates \eqref{est-D} and~\eqref{est-phi_0}. Finally, the $l$-dimensional integral over the Lebesgue measure $dH$ in $\mfa^+$ is, up to a constant, an integral of $h^{l-1}dh$ over $h\in\R$ (as we are integrating radial functions). Altogether, this gives the formula above.
\end{proof}

We consider first the case when $\psi$ vanishes at 0.

\begin{lemma}\label{L1-norm-of-L-psi(0)=0}
In the assumptions of Theorem \ref{eit-psi-LB-no-gap}, suppose that $\psi(0)=0$. Then for $t$ large enough, $k_{t,\cal L}$ is integrable on $G$ and
$$
\| k_{t,\cal L} \|_1 \lesssim |t|^{ \frac {\nu-1}2} (\log |t|)^{3d+3+(3l+1)/2}.
$$
\end{lemma}
\begin{proof}
We can rewrite, when $\tau:=|t|>0$ is large enough, the closure of the positive chamber $\Ca=\bigcup_{i=1}^5 E_i$ with 
\begin{align*}
	E_1&=\set{H\in \Ca: 0\le \abso{H}\le \log \tau };
	\\E_2&=\set{H\in \Ca: \log \tau\le \abso{H}\le \tau-\delta\log \tau };
	\\E_3&=\set{H\in \Ca: \tau-\delta\log \tau\le \abso{H}\le t+\delta\log \tau };
	\\E_4&=\set{H\in \Ca: \tau+\delta\log \tau\le \abso{H}\le 3\tau};
	\\E_5&=\set{H\in \Ca: \abso{H}\ge 3\tau}.
\end{align*}
On $E_1$ we can estimate the kernel by \Cref{eit-psi-L-no-gap} (1), and get
\begin{align*}
	\norm{\chi_{KE_1}k_{t, \cal L}}_1
&\lesssim \int_0^{\log \tau} (1+h)^{\nu+2d} h^{l-1} \tau^{-\nu} dh
	\\&\lesssim \tau^{-\nu} (\log \tau)^{\nu+2d+l} = \tau^{-\nu} (\log \tau)^{2\nu}.
\end{align*}
On $E_2\cup E_4$ we can use \Cref{eit-psi-L-no-gap} (4), with $b=1$. We apply it however with $b=0$, to show that the resulting power of $t$ is the same as in the case when $\psi(0)\ne0$. Denote $I_2 = [\log \tau, \tau-\delta\log \tau]$ and $I_4 = [ \tau+\delta\log \tau, 3\tau]$, then
\begin{align*}
	&\norm{\chi_{K(E_2\cup E_4)}k_{t, \cal L}}_1 
	\\&\lesssim \int_{I_2\cup I_4} (1+h)^d h^{l-1} \big[ h^{-d-l} + h^{(1-l)/2}|h- \tau|^{-d-(l+1)/2} (\log \tau)^{4d+3+2l} \big] dh
	\\&\lesssim \int_1^{3\tau} h^{-1} dh + (\log \tau)^{4d+3+2l} \int_{I_2\cup I_4} h^{d+(l-1)/2}|h- \tau|^{-d-(l+1)/2} dh
	\\&\lesssim \log \tau
    + \tau^{d+(l-1)/2}(\log \tau)^{4d+3+2l}\int_{I_2\cup I_4} |h- \tau|^{-d-(l+1)/2} \, dh.
\end{align*}
For $\a>1$ and large $\tau$, we can estimate
\begin{align*}
    \int_{I_2\cup I_4} |h- \tau|^{-\a} \, dh
    &= -\frac1{1-\a} \Big[ (\tau-h)^{1-\a}\Big]_{\log \tau}^{\tau-\delta\log \tau} + \frac1{1-\a} \Big[ (h-\tau)^{1-\a}\Big]_{\tau+\delta\log \tau}^{3\tau}
    \\&\lesssim (\log \tau)^{1-\a} + \tau^{1-\a} \lesssim (\log \tau)^{1-\a},
\end{align*}
which implies
\begin{equation}\label{E2_E4}
\norm{\chi_{K(E_2\cup E_4)}k_{t, \cal L}}_1 \lesssim \tau^{d+(l-1)/2} (\log \tau)^{3d+3+(3l+1)/2}.
\end{equation}
On $E_3$ with \Cref{eit-psi-L-no-gap} (2) we have
\[
\norm{\chi_{KE_3}k_{t, \cal L}}_1 \lesssim \int_{\tau-\delta \log \tau}^{ \tau+\delta \log \tau} \tau^{(1-l)/2} (\log \tau)^d (1+h)^d h^{l-1} dh
\lesssim \tau^{d+(l-1)/2} (\log \tau)^{d+1} ;
\]
and on $E_5$ with \Cref{eit-psi-L-no-gap} (3), where $b=1$,
\[
\norm{\chi_{KE_5}k_{t, \cal L}}_1 \lesssim \int_{3\tau}^\infty (1+h)^d h^{l-1} h^{-d-l-1} dh\, \lesssim \tau^{-1}.
\]
We stress that the estimates on $E_k$, $k=1,2,3,4$ do not depend on $b$ and are valid also for $b=0$, when $\psi(0)$ is not supposed to be zero.

Collecting the estimates above, we arrive at
\begin{align*}
	\norm{k_{t, \cal L}}_1  \lesssim  \tau^{(2d+l-1)/2} (\log \tau)^{3d+3+(3l+1)/2},
\end{align*}
which completes the proof.
\end{proof}

This theorem applies in particular to any function $\psi$ supported in $(0,+\infty)$, provided it is $\nu$ times differentiable.

\begin{lemma}\label{lemma-e-r}
For $\Psi(r^2) = \exp(itr) e^{-r}$, the kernel $k_{t,\cal L}$ of $\Psi(\sqrt {\cal L})$ has norm
$ \|k_{t,\cal L}\|_1 \lesssim |t|^{(\nu-1)/2} (\log |t|)^{3d+3+(3l+1)/2}$ for $|t|$ large enough.
\end{lemma}
\begin{proof}
We decompose $\Ca$ similarly to Lemma \ref{L1-norm-of-L-psi(0)=0}, with the same $E_k$, $k=1,2,3,4$, but set this time
\begin{align*}
	E_5&=\set{H\in \Ca: 3|t|\le \abso{H}\le e^{|t|}};
	\\E_6&=\set{H\in \Ca: \abso{H}\ge e^{|t|}}.
\end{align*}
As mentioned in Lemma \ref{L1-norm-of-L-psi(0)=0}, on $E=E_1\cup E_2\cup E_3\cup E_4$ we have
\[
\norm{\chi_{KE}k_{t, \cal L}}_1 \lesssim |t|^{(\nu-1)/2} (\log |t|)^{3d+3+(3l+1)/2}.
\]
On $E_5$, we use \Cref{eit-psi-L-no-gap} (3) with $b=0$ and Corollary \ref{integral_formula_L1_norm} to calculate
\[
\norm{\chi_{KE_5}k_{t, \cal L}}_1 \lesssim \int_{3|t|}^{e^{|t|}} h^{d+l-1} h^{-d-l} dh = |t|-\log(3|t|),
\]
which is less than $\norm{\chi_{KE}k_{t, \cal L}}_1$.

And on the set $E_6$, we apply Lemma \ref{lemma-poisson-ptwise}:
$|k_{t,\cal L}(e^H)| \lesssim h^{-d-l-1}\log h $, thus
$$
\|\chi_{E_6} k_{t,\cal L}\|_1 \lesssim \int_{e^{|t|}}^\infty (1+h)^d h^{l-1} h^{-d-l-1} \log h dh \lesssim \int_{e^{|t|}}^\infty h^{-3/2} dh \lesssim e^{-|t|/2},
$$
completing the proof.
\end{proof}

\begin{theorem}\label{L-l1-norm}
Let $\psi$ satisfy the assumptions of \Cref{eit-psi-LB-no-gap}, then for $|t|$ large enough
$$
\| k_{t,\cal L}\|_1 \lesssim |t|^{(\nu-1)/2} (\log |t|)^{3d+3+(3l+1)/2}.
$$
\end{theorem}
\begin{proof}
Let us denote $\psi_1(r) = \psi(r) - \psi(0) e^{-r}$, then $\psi_1$ satisfies conditions of Theorem \ref{L1-norm-of-L-psi(0)=0} and $\psi_1(0)=0$, thus its kernel $k_{1,t}$ has norm $\|k_{1,t}\|_1 \lesssim |t|^{(\nu-1)/2} (\log |t|)^{3d+3+(3l+1)/2}$. But the kernel corresponding to $\psi_2=\psi-\psi_1$ has the same bound by Lemma \ref{lemma-e-r}; this proves the theorem.
\end{proof}

By complex interpolation, we get immediately
\begin{corollary}\label{L-Lp-Lq-norms}
In the assumptions of Theorem \ref{eit-psi-LB-no-gap}, for $t$ is large enough and $1\le p \le \infty$,
    \[
    \norm{\Psi_t(\cal L)}_{L^p(S)\rightarrow L^p(S)} \lesssim  t^{(\nu-1) |\frac 1p -\frac12|} (\log t)^{[6d+3l+7]|\frac 1p -\frac12|}.
    \]
\end{corollary}

\begin{remark}
Theorem \ref{L-l1-norm} and Corollary \ref{L-Lp-Lq-norms} apply clearly to the kernels of $\sin(t \sqrt {\cal L}) \psi( \sqrt {\cal L})$ and $\cos(t \sqrt {\cal L}) \psi( \sqrt {\cal L})$, with the same assumptions on $\psi$.
\end{remark}

\begin{remark}\label{poisson-L-l1-norm}
An important particular case is the Poisson kernel. For $\tau = u+it$ with $u,t$ real, $u>0$, let $k_\tau$ be the kernel of $\exp(-\tau \sqrt{\cal L})$. Theorem \ref{L-l1-norm} gives estimates of $\|k_{1+it}\|_1$. Next, if we applied this theorem to $\tau=u+it$, we would get a positive power of $u$ in the estimates. But we can use results of Anker and Ji \cite[Theorem 4.3.1(ii)]{AnkerJi} instead: for $u\ge1$ and $x\in S$,
$$
|k_u(x)| \asymp u (u+|x|)^{-l-2d-1} \phi_0(x) \delta(x)^{1/2}.
$$
By \eqref{est-phi_0} and Corollary \ref{integral_formula_L1_norm}, this implies (for $u\ge1$)
\begin{align*}
\| k_u \|_1 & \le u \int_0^\infty (u+h)^{-l-2d-1} (1+h)^{2d} h^{l-1} dh
\\ & \lesssim u \Big( u^{ -l-2d-1} \int_0^u (1+h)^{2d} h^{l-1} dh + \int_u^\infty h^{-2} dh \Big)
\\ & \lesssim u^{ -l-2d} \Big( 1 + \int_1^u h^{2d+l-1} dh \Big) + 1
\lesssim 1,
\end{align*}
as in the euclidean case.
Writing $k_{u+it} = k_{u-1} * k_{1+it}$, we can conclude that 
$$
\| k_{u+it} \|_1 \lesssim |t|^{(\nu-1)/2} (\log |t|)^{3d+3+(3l+1)/2},
$$
uniformly in $u\ge 2$ and for $t$ large enough, independently of $u$.
\end{remark}

\bibliographystyle{abbrv} 
\bibliography{references}



%
%
%

\end{document}